\documentclass[onefignum,onetabnum]{siamart190516}

\usepackage{lipsum}
\usepackage{amsfonts}
\usepackage{graphicx}
\usepackage{subfigure}
\usepackage{epstopdf}
\usepackage{algorithmic}
\ifpdf
  \DeclareGraphicsExtensions{.eps,.pdf,.png,.jpg}
\else
  \DeclareGraphicsExtensions{.eps}
\fi

\newcommand{\N}{\mathbb{N}}
\newcommand{\E}{\mathbb{E}}
\newcommand{\R}{\mathbb{R}}
\newcommand{\T}{\mathbb{T}}
\renewcommand{\L}{\mathcal{L}}
\newcommand{\C}{\mathcal{C}}

\newcommand{\scal}[2]{#1 \cdot #2}
\newcommand{\abs}[1]{\big| #1 \big|}
\newcommand{\norm}[1]{\big\| #1 \big\|}

\newsiamremark{remark}{Remark}
\newsiamremark{hypothesis}{Hypothesis}
\crefname{hypothesis}{Hypothesis}{Hypotheses}
\newsiamthm{claim}{Claim}

\newsiamthm{theo}{Theorem}
\newsiamthm{propo}{Proposition}
\newsiamthm{hyp}{Assumption}
\newsiamremark{rem}{Remark}
\newsiamremark{remarks}{Remarks}
\newsiamremark{example}{Example}

\headers{On Asymptotic Preserving schemes for a class of SDEs}{C.-E.~Br\'ehier, S.~Rakotonirina-Ricquebourg}

\title{On Asymptotic Preserving schemes for a class of Stochastic Differential Equations in averaging and diffusion approximation regimes
}

\author{Charles-Edouard Br\'ehier\thanks{Univ Lyon, Universit\'e Claude Bernard Lyon 1, CNRS UMR 5208, Institut Camille Jordan, 43 blvd. du 11 novembre 1918, F-69622 Villeurbanne cedex, France
  (\email{brehier@math.univ-lyon1.fr},\email{rakoto@math.univ-lyon1.fr}).}
\and Shmuel Rakotonirina-Ricquebourg\footnotemark[1]
}

\usepackage{amssymb}
\usepackage{amscd}

\usepackage{amsopn}

\ifpdf
\hypersetup{
  pdftitle={On Asymptotic Preserving schemes for a class of Stochastic Differential Equations in averaging and diffusion approximation regimes},
  pdfauthor={C.-E.~Br\'ehier, S.~Rakotonirina-Ricquebourg}
}
\fi

\begin{document}

\maketitle

\begin{abstract}
We introduce and study a notion of Asymptotic Preserving schemes, related to convergence in distribution, for a class of slow-fast Stochastic Differential Equations. In some examples, crude schemes fail to capture the correct limiting equation resulting from averaging and diffusion approximation procedures. We propose examples of Asymptotic Preserving schemes: when the time-scale separation vanishes, one obtains a limiting scheme, which is shown to be consistent in distribution with the limiting Stochastic Differential Equation. Numerical experiments illustrate the importance of the proposed Asymptotic Preserving schemes for several examples. In addition, in the averaging regime, error estimates are obtained and the proposed scheme is proved to be uniformly accurate.
\end{abstract}



\section{Introduction}\label{sec:intro}

Deterministic and stochastic systems are ubiquitous in science and engineering. Traditional modelling and numerical methods become ineffective when systems evolve at different time scales: see for instance the monographs~\cite{E:11,Kuehn:15} for comprehensive treatment of multiscale dynamics. Averaging and homogenization~\cite{PavliotisStuart} are two popular techniques which are employed to rigorously derive macroscopic limiting equations, starting from (stochastic) slow-fast systems with separated time-scales.

In the last two decades, constructing efficient numerical methods for multiscale stochastic systems has been a very active research area: let us mention the Heterogeneous Multiscale Method (see~\cite{AbdulleEEngquistVandenEinjden,Brehier:13,ELiuVandenEijnden:05}), projective integration (see~\cite{GivonKevrekidisKupferman:06}), equation-free coarse-graining (see~\cite{KevrekidisAl:03}), spectral methods (see~\cite{AbdullePavliotisVaes:17}), micro-macro acceleration methods (see~\cite{VandecasteeleZielinskiSamaey:20}), parareal algorithms (see~\cite{LegollLelievreMyerscoughSamaey:20}). In the methods mentioned above, the objective is to approximate the limiting model for the slow variables of interest, and only partial but relevant information coming from the fast dynamics is taken into account. As a consequence, these methods may not be appropriate if one wants to approximate simultaneously the original multiscale model and its limit. In this article, we focus on the notion of asymptotic preserving schemes, in order to overcome this issue.

To motivate and illustrate our work, let us introduce simplified versions of the systems of Stochastic Differential Equations (SDE) considered in this article. The time-scale separation parameter is denoted by $\epsilon\in(0,1]$. On the one hand, in the averaging regime (see Equation~\eqref{eq:SDE-av} in Section~\ref{sec:models-av} for the more general version), we consider systems of the type
\begin{equation}\label{eq:intro-SDE-av}
\left\lbrace
\begin{aligned}
    dX^\epsilon_t &= b(X^\epsilon_t,m^\epsilon_t) dt,\\
    dm^\epsilon_t &= - \frac{m_t^\epsilon}{\epsilon} dt + \frac{\sqrt{2}}{\sqrt{\epsilon}} d\beta_t.
\end{aligned}
\right.
\end{equation}
When $\epsilon\to 0$, the averaging principle (see~\cite[Chapter~$10$]{PavliotisStuart}) states that $X^\epsilon$ converges (at least in distribution) to the solution $X$ of the Ordinary Differential Equation $\dot{X}=\overline{b}(X)$ where $\overline{b}(x)=\int b(x,m)d\nu(m)$ and $\nu=\mathcal{N}(0,1)$ is the standard Gaussian random variable. On the other hand, in the diffusion approximation regime (see Equation~\eqref{eq:SDE-diff} in Section~\ref{sec:models-diff} for the more general version), we consider systems of the type
\begin{equation}\label{eq:intro-SDE-diff}
\left\lbrace
\begin{aligned}
dX_t^\epsilon &= \frac{\sigma(X_t^\epsilon)m_t^\epsilon}{\epsilon} dt,\\
dm^\epsilon_t &= - \frac{m_t^\epsilon}{\epsilon^2} dt + \frac{1}{\epsilon} d\beta_t.
\end{aligned}
\right.
\end{equation}
When $\epsilon\to 0$, the diffusion approximation result (see~\cite[Chapter~$11$]{PavliotisStuart}) states that $X^\epsilon$ converges (in distribution) to the solution $X$ of the SDE
\[
dX_t=\sigma(X_t)\circ d\beta_t,
\]
where the noise is interpreted in the Stratonovich sense. This type of results is related to results known as Wong-Zakai approximation and Smoluchowski-Kramers limits in the literature. In the two SDE systems~\eqref{eq:intro-SDE-av} and~\eqref{eq:intro-SDE-diff}, the fast component is an Ornstein-Uhlenbeck process.

In this article, we are interested in the behavior when $\epsilon\to 0$ of numerical schemes for the SDEs~\eqref{eq:intro-SDE-av} and~\eqref{eq:intro-SDE-diff}. To explain the challenge faced and the solutions proposed in this article, we consider the following schemes, which are both consistent for any fixed value of $\epsilon>0$. On the one hand, in the averaging regime one defines
\begin{equation}\label{eq:intro-nonAP-av}
    \left\lbrace
    \begin{aligned}
        X^\epsilon_{n+1} &= X^\epsilon_n + \Delta t b(X^\epsilon_n,m^\epsilon_{n+1}),\\
        m^\epsilon_{n+1} &= m_n^\epsilon-\frac{\Delta t}{\epsilon}m^\epsilon_{n+1}+\sqrt{\frac{2\Delta t}{\epsilon}}\gamma_n.
    \end{aligned}
    \right.
\end{equation}
On the other hand, in the diffusion approximation regime one defines
\begin{equation}\label{eq:intro-nonAP-diff}
\left\lbrace
\begin{aligned}
X_{n+1}^\epsilon&=X_n^\epsilon+\sigma(X_n^\epsilon)\frac{\Delta tm_{n+1}^\epsilon}{\epsilon},\\
m_{n+1}^\epsilon&=m_n^\epsilon-\frac{\Delta t}{\epsilon^2}m_{n+1}^\epsilon+\frac{\sqrt{\Delta t}}{\epsilon}\gamma_n,\end{aligned}
\right.
\end{equation}
In the schemes~\eqref{eq:intro-nonAP-av} and~\eqref{eq:intro-nonAP-diff}, $\bigl(\gamma_n\bigr)_{n\ge 0}$ is a sequence of independent standard Gaussian random variables. One may check that $X_n^\epsilon\to X_n$ for all $n\ge 0$, in probability, when $\epsilon\to 0$, where the limiting schemes are given by
\begin{equation*}
X_{n+1}=X_n+\Delta tb(X_n,0)
\end{equation*}
in the averaging regime, and
\begin{equation*}
X_{n+1}=X_n+\sqrt{\Delta t}\sigma(X_n)\gamma_n,
\end{equation*}
in the diffusion approximation regime. Note that, in the second case, the limiting scheme is consistent with the It\^o interpretation of the noise, instead of the correct Stratonovich one. In the two cases, the limiting scheme is in general not consistent with the limiting equation, and using such a scheme in practice may lead to drawing false conclusions about the limiting system from numerical experiments. We refer to~\cite{FrankGottwald:18,LiAbdulleE:08} for other examples of situations where numerical schemes perform badly when applied to multiscale SDE systems.

The objective of this article is to design and study Asymptotic Preserving (AP) schemes, such that the following diagram commutes (where convergence is understood in distribution): if $T=N\Delta t$, one has
\[
\begin{CD}
X_N^\epsilon     @>{\Delta t\to  0}>> X^\epsilon(T) \\
@VV{\epsilon\to 0}V        @VV{\epsilon\to 0}V\\
\overline{X}_N     @>{\Delta t\to 0}>> \overline{X}(T)
\end{CD}
\]
The two schemes~\eqref{eq:intro-nonAP-av} and~\eqref{eq:intro-nonAP-diff} described above are not AP. The notion of AP schemes has been introduced in~\cite{Jin:99}, for applications to multiscale kinetic Partial Differential Equations (PDEs), which converge to parabolic diffusion PDEs. We refer to\cite[Section~7]{DimarcoPareschi:14},~\cite{HuJinLi:17},~\cite{Jin:12} and~\cite[Section~4]{Puppo:19} for recent reviews on AP schemes for this type of models. To the best of our knowledge, the design and analysis of Asymptotic Preserving schemes for slow-fast SDEs of the type~\eqref{eq:intro-SDE-av} and~\eqref{eq:intro-SDE-diff} has not been considered so far in the literature. Note that a specific feature (compared with the deterministic case) is the need to consider convergence in distribution. Let us mention related works for Stochastic Partial Differential Equations (SPDEs), in the diffusion approximation regime. First, in~\cite{DuboscqMarty:16,Marty:06}, the authors consider Schr\"odinger equations and study an abstract asymptotic preserving property. However, they do not propose implementable schemes. In~\cite{AyiFaou:19}, the authors deal with some multiscale stochastic kinetic PDEs, driven by a Wiener process. However, the structure of the model is different from the one of~\eqref{eq:intro-SDE-diff}. In a future work~\cite{BrehierHivertRakotonirina}, we plan to apply the findings of this article to the SPDE models considered in~\cite{AyiFaou:19}. The works mentioned above concerning SPDE models are limited to diffusion coefficients of the type $\sigma(x)=x$, for which specific arguments may give a straightforward construction of AP schemes, for appropriate discretization of the fast component. An AP scheme in the case $\sigma(x)=1$ for~\eqref{eq:intro-SDE-diff} is proposed in~\cite{PavliotisStuartZygalakis:09}, however the subtlety of the interpretation of the noise at the limit is not relevant in that case. Finally, let us also mention that AP schemes have also been studied for PDEs with random coefficients, see~\cite{HuJin:17,Jin:18,JinLuPareschi:18} or in the context of Monte-Carlo methods for deterministic problems, see~\cite{DimarcoPareschiSamaey:18,RenLuiJin:14}.

We are now in position to describe the contributions of this article. In Section~\ref{sec:num-gen}, we define the appropriate notion of AP schemes for SDE systems, related to convergence in distribution, and study several general properties.

Our first main result is Theorem~\ref{th:AP-av}, which exhibits an example of AP scheme in the averaging regime: for the simplified version~\eqref{eq:intro-SDE-av}, the scheme is given by
\begin{equation}\label{eq:intro-AP-av}
    \left\lbrace
    \begin{aligned}
        X^\epsilon_{n+1} &= X^\epsilon_n + \Delta t b(X^\epsilon_n,m^\epsilon_{n+1}),\\
        m^\epsilon_{n+1} &= e^{-\frac{\Delta t}{\epsilon}} m^\epsilon_n + \sqrt{1 - e^{-\frac{2\Delta t}{\epsilon}}}  \gamma_n.
    \end{aligned}
    \right.
\end{equation}
The fast component in the scheme above is discretized using a scheme which is exact in distribution.

Our second main result is Theorem~\ref{th:UA-av}, which states error estimates of the type
\begin{equation*}
\underset{\epsilon\in(0,1]}\sup~\big|\E[\varphi(X_N^\epsilon)-\E[\varphi(X^\epsilon(T))]\big|={\rm O}\bigl(\sqrt{\Delta t}\bigr),
\end{equation*}
for sufficiently smooth real-valued mappings $\varphi$. This error estimate means that the scheme is Uniformly Accurate.

Finally, our third main result is Theorem~\ref{th:AP-diff}, which exhibits an example of AP scheme in the diffusion approximation regime: for the simplified version~\eqref{eq:intro-SDE-diff} (see Corollary~\ref{cor:diff-ex1}), the scheme is given by
\begin{equation}\label{eq:intro-AP-diff}
\left\lbrace
\begin{aligned}
m_{n+1}^\epsilon&=m_n^\epsilon-\frac{\Delta t}{\epsilon^2}m_{n+1}^\epsilon+\frac{\sqrt{\Delta t}}{\epsilon}\gamma_n,\\
Y_{n+1}^\epsilon&=X_n^\epsilon+\sigma(X_n^\epsilon)\frac{\Delta t m_{n+1}^\epsilon}{\epsilon},\\
X_{n+1}^\epsilon&=X_n^\epsilon+\frac{\sigma(X_n^\epsilon)+\sigma(Y_{n+1}^\epsilon)}{2}\frac{\Delta tm_{n+1}^\epsilon}{\epsilon}.
\end{aligned}
\right.
\end{equation}
A prediction-correction method is employed to retrieve the correct interpretation of the noise for the limiting equation: the scheme~\eqref{eq:intro-SDE-diff} is indeed consistent with the Stratonovich interpretation of the noise.

Let us also mention that another situation is considered in Corollary~\ref{cor:diff-ex2}: for the model~\eqref{eq:SDE-diff-ex2} taken from~\cite{LaibeBrehierLombart:20} (with an application in astrophysics), the limiting equation~\eqref{eq:SDE-limit-diff-ex2} contains a so-called noise-induced drift-term, which is captured only for well-designed AP schemes.

Some numerical experiments (see Section~\ref{sec:exp}) show that the AP schemes~\eqref{eq:intro-AP-av} and~\eqref{eq:intro-AP-diff} are effective in all regimes $\epsilon>0$ and $\epsilon\to 0$, contrary to the schemes~\eqref{eq:intro-nonAP-av} and~\eqref{eq:intro-nonAP-diff} which fail to capture the correct limiting behavior when $\epsilon\to 0$.

The article is organized as follows. The general SDE models in the averaging and diffusion approximation regimes are presented in Sections~\ref{sec:models-av} and~\ref{sec:models-diff}. The main results of this article are stated in Section~\ref{sec:num}: the general theory of AP schemes is presented in Section~\ref{sec:num-gen}, and it is applied in the averaging and diffusion approximation regimes in Section~\ref{sec:num-av} and~\ref{sec:num-diff} respectively. Numerical experiments are reported in Section~\ref{sec:exp}. Section~\ref{sec:error} is devoted to the proof of the error estimates stated in Theorem~\ref{th:UA-av}. Finally, Section~\ref{sec:conclusion} gives some conclusions and perspectives.

\section{Slow-fast SDE models and their limits}\label{sec:models}

Without loss of generality, the time-scale separation parameter $\epsilon$ satisfies $\epsilon\in(0,1]$. The time-step size of the integrators studied in this work is denoted by $\Delta t$. It is assumed that $\Delta t=\frac{T}{N}$ where $T\in(0,\infty)$ is a fixed time and $N\in\N$. Without loss of generality, it is assumed that $\Delta t\in(0,1]$.

In the slow-fast systems considered in this work, the slow component $X^\epsilon$ takes values in the $d$-dimensional flat torus $\T^d$, where $d\in\N$ is an arbitrary integer, whereas the fast component $m^\epsilon$ takes values in $\R$. The framework and the models considered in this work may be generalized in many ways to more complex situations, however the arguments and results below are sufficient to illustrate the difficulties of designing asymptotic preserving schemes for stochastic equations.

Let $\bigl(\beta_t\bigr)_{t\ge 0}$ and $\bigl(B_t\bigr)_{t\ge 0}$ be two independent standard Wiener processes, with values in $\R$ and $\R^D$ respectively, where $D\in\N$, defined on a probability space $(\Omega,\mathcal{F},\mathbb{P})$ which satisfies the usual conditions.

The following notation for derivatives is used below: $\nabla_x=\bigl(\partial_{x_i} \bigr)_{1\le i\le d}\in \R^d$ and $\partial_m$ are the partial gradient and derivative operators with respect to $x$ and $m$ respectively. If $\sigma$ is a mapping with values in $\mathcal{M}_{d,D}(\R)$ (the space of $d\times D$ matrices with real entries), let $\sigma^\star$ denote the transpose of $\sigma$, and set $\sigma \sigma^\star : \nabla_x^2=\sum_{i,j=1}^{d}\bigl(\sigma\sigma^\star)_{i,j} \partial_{x_i}\partial_{x_j}$. If $b$ is a $\R^d$-valued mapping, let $\scal{b}{\nabla_x}=\sum_{i=1}^{d}b_i \partial_{x_i}$.

\begin{hyp}\label{ass:init}
The initial conditions $X_0^\epsilon\in\T^d$ and $m_0^\epsilon\in\R$ of the processes are deterministic quantities and they satisfy
\[
X_0^\epsilon=x_0^\epsilon \xrightarrow[\epsilon\to 0]{} x_0~,\quad \sup_{\epsilon\in (0,1]}|m_0^\epsilon|<\infty.
\]
\end{hyp}

\subsection{The averaging regime}\label{sec:models-av}

In the so-called averaging regime, we consider slow-fast SDE systems of the type
\begin{equation}\label{eq:SDE-av}
\left\lbrace
\begin{aligned}
    dX^\epsilon_t &= b(X^\epsilon_t,m^\epsilon_t) dt + \sigma(X^\epsilon_t,m^\epsilon_t) dB_t,\\
    dm^\epsilon_t &= - \frac{m_t^\epsilon}{\epsilon} dt + \frac{\sqrt{2}h(X^\epsilon_t)}{\sqrt{\epsilon}} d\beta_t.
\end{aligned}
\right.
\end{equation}

The coefficients appearing in~\eqref{eq:SDE-av} are assumed to satisfy the following conditions.
\begin{hyp}\label{ass:coeffs-av}
    The functions $b:\T^d\times\R\to\R^d$ and $\sigma:\T^d\times\R\to\mathcal{M}_{d,D}(\R)$ are assumed to be of class $\C^4$, and $h:\T^d\to\R$ is assumed to be of class $\C^1$. Moreover, they are all assumed to be bounded and to have bounded derivatives.
\end{hyp}

Owing to Assumption~\ref{ass:coeffs-av}, for all initial conditions $X_0^\epsilon\in\T^d$ and $m_0^\epsilon\in\R$, and for every $\epsilon\in(0,1]$, there exists a unique global solution $\bigl(X^\epsilon(t),m^\epsilon(t)\bigr)_{t\ge 0}$ of the SDE system~\eqref{eq:SDE-av}. Since $h$ is bounded, it is straightforward to check that
\begin{equation}\label{eq:moment_m-eps-t}
    \sup_{\epsilon \in (0,1]} \sup_{t \geq 0} \frac{\E[|m^\epsilon(t)|^2]}{1+|m_0^\epsilon|^2}<\infty.
\end{equation}
This estimate will prove useful to prove Proposition~\ref{propo:limit-SDE-av}.

The infinitesimal generator $\L^\epsilon$ associated with the SDE~\eqref{eq:SDE-av} has the following expression:
\begin{equation}\label{eq:L^epsilon-av}
\L^\epsilon = \frac{1}{\epsilon} \L_{OU} + \L_0,
\end{equation}
where
\begin{equation}\label{eq:L_01-av}
    \begin{aligned}
        \L_0 &= \scal{b(x)}{\nabla_x}+\frac{1}{2}\sigma\sigma^\star:\nabla_x^2,\\
        \L_{OU} &= - m \partial_m + h(x)^2 \partial_m^2,\\
    \end{aligned}
\end{equation}
Observe that for fixed $x\in\T^d$, $\L_{OU}$ is the generator of an ergodic Ornstein-Uhlenbeck process. The associated invariant distribution is $\nu^x=\mathcal{N}(0,h(x)^2)$.

Define averaged coefficients as follows: for all $x\in \T^d$
\begin{equation}\label{eq:averagedcoeffs1}
    \overline{b}(x)=\int b(x,m)d\nu^x(m)~,\quad \overline{a}(x)=\int \sigma(x,m)\sigma(x,m)^\star d\nu^x(m).
\end{equation}
Note that $\overline{b}:\T^d\to\R^d$ is of class $\C^4$. The averaging principle result stated below requires the following condition to be satisfied.
\begin{hyp}\label{ass:av}
    There exists an integer $\overline{D}\in\N$ and a function $\overline{\sigma}:\T^d\to \mathcal{M}_{d,\overline{D}}(\R)$ of class $\C^4$ such that for all $x\in\T^d$
    \begin{equation}\label{eq:averagedcoeffs2}
        \overline{a}(x)=\overline{\sigma}(x)\overline{\sigma}(x)^\star.
    \end{equation}
\end{hyp}
Assumption~\ref{ass:av} holds if there exists $c\in(0,\infty)$ such that $\overline{a}(x)\ge cI$ for all $x\in\T^d$ (as symmetric matrices). This condition is satisfied when $\sigma$ only depends on the slow variable $x$ ($\partial_m\sigma(x,m)=0$ for all $(x,m)\in\T^d\times\R$), or when $\sigma(x,m)\sigma(x,m)^\star\ge cI$ for all $(x,m)\in\T^d\times\R$. In that case, one can choose $\overline{D}=d$. If the diffusion coefficient is of the type $\sigma(x,m)=\sigma^\sharp(m)\sigma^\dagger(x)$, with $\sigma^\dagger(x)\in\R^d$ and $\sigma^\sharp(m)\in\R$, then one can choose $\overline{D}=D$ and $\overline{\sigma}(x)=\sigma^\dagger(x) \sqrt{\int \sigma^\sharp(m)^2 d\nu^x(m)}$ for all $x\in\T^d$.

We are now in position to state the averaging principle result and to define the limiting process $X$ obtained when $\epsilon\to 0$.
\begin{propo}\label{propo:limit-SDE-av}
    Let Assumptions~\ref{ass:init},~\ref{ass:coeffs-av} and~\ref{ass:av} be satisfied. Let $T\in(0,\infty)$. When $\epsilon\to 0$, the $\mathcal{C}([0,T],\T^d)$-valued process $\bigl(X^\epsilon(t)\bigr)_{0\le t\le T}$ converges in distribution to the solution $\bigl(X(t)\bigr)_{0\le t\le T}$ of the limiting SDE
    \begin{equation}\label{eq:limit_SDE-av}
        dX_t=\overline{b}(X_t)dt+\overline{\sigma}(X_t)d\overline{B}_t,
    \end{equation}
    with initial condition $X(0)=x_0$, where the coefficients $\overline{b}$ and $\overline{\sigma}$ are defined by~\eqref{eq:averagedcoeffs1}--\eqref{eq:averagedcoeffs2}, and where $\bigl(\overline{B}_t\bigr)_{t\ge 0}$ is a standard $\R^{\overline{D}}$-valued Wiener process.

    The infinitesimal generator $\L$ associated with the limiting SDE~\eqref{eq:limit_SDE-av} is given by
    \begin{equation}\label{eq:limit_generator-av}
        \L=\scal{\overline{b}(x)}{\nabla_x}+\frac{1}{2}\overline{\sigma}~\overline{\sigma}^\star:\nabla_x^2,
    \end{equation}
    and is such that the following property holds: let $\varphi \in \C^4(\T^d)$, then there exists a function $\varphi^1:\T^d\times\R\to\R$ such that
    \begin{align}
        \varphi^\epsilon&=\varphi+\epsilon\varphi_1, \label{eq:def_phi-epsilon-av}\\
        \L^\epsilon \varphi^\epsilon &\xrightarrow[\epsilon\to 0]{} \L\varphi. \label{eq:generator_convergence-av}
    \end{align}

    Finally, let $\varphi \in \C^4(\T^d)$, then there exists $C(T,\varphi)\in(0,\infty)$ such that
    \begin{equation}\label{eq:error-av}
        \big|\E[\varphi(X^\epsilon(T))]-\E[\varphi(X(T))]\big|\le C(T,\varphi)\epsilon.
    \end{equation}
\end{propo}
The averaging principle stated in Proposition~\ref{propo:limit-SDE-av} is a standard result, see for instance~\cite[Chapter~16]{PavliotisStuart}. In general the convergence stated in Proposition~\ref{propo:limit-SDE-av} only holds in distribution, however it holds in stronger sense (for instance in mean-square sense) if $\sigma$ only depends on $x$.

We refer to Appendix~\ref{sec:app-av} for a sketch of the construction of the perturbed test function $\varphi^\epsilon$ which satisfies~\eqref{eq:def_phi-epsilon-av}--\eqref{eq:generator_convergence-av} (see~\cite[Chapter~6]{FouqueGarnierPapanicolaouSolna:07} for a detailed description of the perturbed test function method). Note that the perturbed test function appears in Proposition~\ref{propo:AP} below. For the error estimate~\eqref{eq:error-av}, see Lemma~\ref{lem2} and its proof below.

\subsection{The diffusion approximation regime}\label{sec:models-diff}

\subsubsection{General model}\label{sec:models-diff-gen}

In the so-called diffusion approximation regime, we consider slow-fast SDE systems of the type
\begin{equation}\label{eq:SDE-diff}
    \left\{
    \begin{aligned}
        dX_t^\epsilon &= b(X_t^\epsilon) dt + \frac{\sigma(X_t^\epsilon)m_t^\epsilon}{\epsilon} dt,\\
        dm_t^\epsilon&=f(X_t^\epsilon)\Bigl(-\frac{m_t^\epsilon}{\epsilon^2}dt+\frac{g(X_t^\epsilon)}{\epsilon}dt+\frac{h(X_t^\epsilon)}{\epsilon}d\beta_t\Bigr).
    \end{aligned}
    \right.
\end{equation}

The coefficients appearing in~\eqref{eq:SDE-diff} are assumed to satisfy the following conditions.
\begin{hyp}\label{ass:coeffs-diff}
The functions $b:\T^d\to\R^d$ and $g,h:\T^d\to\R$ are assumed to be of class $\mathcal{C}^1$. The functions $\sigma:\T^d\to\R^d$ and $f:\T^d\to\R$ are assumed to be of class $\mathcal{C}^2$. Moreover, $f$ takes values in $(0,\infty)$: we assume that ${\displaystyle \min_{x \in \mathbb T^d} f(x) > 0}$.
\end{hyp}
Owing to Assumption~\ref{ass:coeffs-diff}, for all initial conditions $X_0^\epsilon\in\T^d$ and $m_0^\epsilon\in\R$, and for every $\epsilon\in(0,1]$, there exists a unique global solution $\bigl(X^\epsilon(t),m^\epsilon(t)\bigr)_{t\ge 0}$ of the SDE system~\eqref{eq:SDE-diff}. The infinitesimal generator $\L^\epsilon$ associated with the SDE~\eqref{eq:SDE-diff} has the following expression:
\begin{equation}\label{eq:L^epsilon-diff}
\L^\epsilon = \frac{1}{\epsilon^2} \L_{OU} + \frac{1}{\epsilon} \L_1 + \L_0,
\end{equation}
where
\begin{equation}\label{eq:L_012-diff}
    \begin{aligned}
        \L_0 &= \scal{b(x)}{\nabla_x},\\
        \L_1 &= \scal{m \sigma(x)}{\nabla_x} + f(x) g(x) \partial_m,\\
        \L_{OU} &= - f(x) m \partial_m + \frac 1 2 f(x)^2 h(x)^2 \partial_m^2.
    \end{aligned}
\end{equation}
Observe that for fixed $x\in\T^d$, $\L_{OU}$ is the generator of an ergodic Ornstein-Uhlenbeck process. The associated invariant distribution is $\nu^x=\mathcal{N}(0,\frac{f(x)h(x)^2}{2})$.

We are now in position to state the diffusion approximation result and to define the limiting process $X$ obtained when $\epsilon\to 0$.
\begin{propo}\label{propo:limit-SDE-diff}
Let Assumptions~\ref{ass:init} and~\ref{ass:coeffs-diff} be satisfied. Let $T\in(0,\infty)$. When $\epsilon\to 0$, the $\mathcal{C}([0,T],\T^d)$-valued process $\bigl(X^\epsilon(t)\bigr)_{0\le t\le T}$ converges in distribution to the solution $\bigl(X(t)\bigr)_{0\le t\le T}$ of the limiting SDE
\begin{equation}
 \label{eq:limit_SDE-diff}
    dX_t = \left( b + g \sigma + \frac{h^2}{2}(\scal{\sigma}{\nabla_x})\sigma - \frac{h^2}{2f}\scal{\sigma}{\nabla_x f}\sigma\right) (X_t) dt + h(X_t) \sigma(X_t) dW_t,
\end{equation}
driven by a standard one-dimensional Wiener process $\bigl(W(t)\bigr)_{t\ge 0}$, with initial condition $X(0)=x_0$.

The infinitesimal generator $\L$ associated with the limiting SDE~\eqref{eq:limit_SDE-diff} is given by
\begin{align}\label{eq:limit_generator-diff}
    \L \varphi &= \scal{\left(b + g \sigma\right)}{\nabla_x \varphi} + \scal{\frac{h^2 f \sigma}{2}}{\nabla_x \left( \scal{\frac{\sigma}{f}}{\nabla_x \varphi} \right)}\\
    &=\scal{\left(b + g \sigma\right)}{\nabla_x \varphi}+\frac{h^2}{2}\sigma \sigma^\star : \nabla_x^2\varphi \nonumber\\
    &\phantom{= }+\frac{h^2}{2} \scal{(\scal{\sigma}{\nabla_x})\sigma}{\nabla_x \varphi}-\frac{h^2}{2f}\scal{\sigma}{\nabla_x f}\scal{\sigma}{\nabla_x\varphi},\nonumber
\end{align}
and is such that the following property holds: let $\varphi \in \C^3(\T^d)$, then one constructs two functions $\varphi_1,\varphi_2:\T^d\times\R\to\R$, such that
\begin{align}
    \varphi^\epsilon&=\varphi+\epsilon\varphi_1+\epsilon^2\varphi_2, \label{eq:def_phi-epsilon-diff}\\
    \L^\epsilon \varphi^\epsilon &\xrightarrow[\epsilon\to 0]{} \L\varphi. \label{eq:generator_convergence-diff}
\end{align}

Finally, let $\varphi\in\C^3(\T^d)$, then there exists $C(T,\varphi)\in(0,\infty)$ such that
\begin{equation}\label{eq:error-diff}
\big|\E[\varphi(X^\epsilon(T))]-\E[\varphi(X(T))]\big|\le C(T,\varphi)\epsilon.
\end{equation}
\end{propo}
The diffusion approximation stated in Proposition~\ref{propo:limit-SDE-diff} is a standard result, see for instance~\cite[Chapter~18]{PavliotisStuart}. We refer to Appendix~\ref{sec:app-diff} for a sketch of the construction of the perturbed test function $\varphi^\epsilon$ which satisfies~\eqref{eq:def_phi-epsilon-diff}--\eqref{eq:generator_convergence-diff} (see~\cite[Chapter~6]{FouqueGarnierPapanicolaouSolna:07} for a detailed description of the perturbed test function method). Since the error estimate~\eqref{eq:error-diff} plays no role in the sequel, the proof is omitted. We refer to~\cite{KhasminskiiYin:05} for arguments using asymptotic expansions of solutions of Kolmogorov equations leading to~\eqref{eq:error-diff}, (see also~\cite{LaibeBrehierLombart:20} for related computations).

\subsubsection{Two examples in the approximation-diffusion regime}\label{sec:models-diff-ex}

The setting described above encompasses several interesting examples of SDE systems. In order to focus on the different possible issues which need to be overcome when constructing asymptotic preserving numerical schemes in the regime $\epsilon\to 0$, we deal with two examples described below. In addition, the asymptotic preserving numerical schemes will have simpler formulations for these examples than in the general case. In both examples, dimension is set equal to $d=1$ to simplify the presentation, and $b=0$.

Let us present the first example: consider the system
\begin{equation}\label{eq:SDE-diff-ex1}
\left\lbrace
\begin{aligned}
dX_t^\epsilon&=\frac{\sigma(X_t^\epsilon)m_t^\epsilon}{\epsilon}dt,\\
dm_t^\epsilon&=-\frac{m_t^\epsilon}{\epsilon^2}dt+\frac{1}{\epsilon}d\beta_t,
\end{aligned}
\right.
\end{equation}
where the coefficients in the fast equation are constant: $f(x)=h(x)=1$ and $g(x)=0$ for all $x\in\T$. Applying Proposition~\ref{propo:limit-SDE-diff} in this example yields the following limiting equation
\begin{equation}\label{eq:SDE-limit-diff-ex1}
dX_t=\sigma(X_t)\circ dW_t,
\end{equation}
where the noise is interpreted using the Stratonovich convention. With the It\^o convention, the equation is written as
\[
dX_t=\frac12\sigma(X_t)\sigma'(X_t)dt+\sigma(X_t)dW_t.
\]
Note that the diffusion approximation result (Proposition~\ref{propo:limit-SDE-diff}) may be obtained by straightforward arguments in two cases, which will be repeated at the discrete-time levels. Let $\zeta^\epsilon(t)=\frac{1}{\epsilon}\int_0^tm^\epsilon(s)ds$ for all $t\ge 0$. First, if $\sigma(x)=1$ for all $x\in\T$, then one has $dX_t^\epsilon=d\zeta_t^\epsilon$. Therefore passing to the limit yields
\[
X^\epsilon(t)=X_0^\epsilon+\zeta^\epsilon(t)\xrightarrow[\epsilon\to 0]{} x_0+W(t),
\]
and the limiting equation is $dX_t=dW_t$. Second, assume that $x$, $X^\epsilon(t)$ and $X(t)$ take values in the real line $\R$ (instead of the torus $\T$) and that $\sigma(x)=x$ for all $x\in\R$. Then~\eqref{eq:SDE-limit-diff-ex1} is written as $dX_t^\epsilon=X_t^\epsilon d\zeta_t^\epsilon$. Computing the solution and passing to the limit then yields
\[
X^\epsilon(t)=X_0^\epsilon \exp\bigl(\zeta^\epsilon(t)\bigr)\xrightarrow[\epsilon\to 0]{} x_0\exp(W(t))=X(t),
\]
and the limiting equation is $dX_t=X_t\circ dW_t$.

Note that when the function $\sigma$ is not constant, the It\^o and Stratonovich interpretations differ. Constructing an asymptotic preserving requires to capture the correction term in a limiting scheme (which will naturally be associated with an It\^o interpretation of the noise).

Let us now present the second example, taken from~\cite{LaibeBrehierLombart:20}. The coefficients $f,g,h$ are allowed to depend on the slow component $x$, whereas it is assumed that $\sigma(x)=1$ for all $x\in\T$. Therefore, the system in the second example has the following expression
\begin{equation}\label{eq:SDE-diff-ex2}
\left\lbrace
\begin{aligned}
dX_t^\epsilon&=\frac{m_t^\epsilon}{\epsilon}dt,\\
dm_t^\epsilon&=f(X_t^\epsilon)\Bigl(-\frac{m_t^\epsilon}{\epsilon^2}dt+\frac{g(X_t^\epsilon)}{\epsilon}dt+\frac{h(X_t^\epsilon)}{\epsilon}d\beta_t\Bigr),
\end{aligned}
\right.
\end{equation}
Applying Proposition~\ref{propo:limit-SDE-diff} in this example yields the following limiting equation
\begin{equation}\label{eq:SDE-limit-diff-ex2}
    dX_t = g(X_t) dt - \frac{h(X_t)^2 f'(X_t)}{2 f(X_t)} dt + h(X_t) dW_t.
\end{equation}
The noise is interpreted in the It\^o sense. Observe that when $f$ is not constant, the noise-induced drift term $\frac{h^2 f'}{2f}$ appears. The construction of asymptotic preserving schemes for this problem requires to be careful in order to capture this additional drift term in the limiting scheme.

\section{Numerical discretization and asymptotic preserving schemes}\label{sec:num}

The objective of this section is to study the notion of Asymptotic Preserving (AP) schemes for the slow-fast SDE system~\eqref{eq:SDE-av} (averaging regime) or~\eqref{eq:SDE-diff} (diffusion approximation regime) when $\epsilon\to 0$. The fundamental requirements to have an AP scheme are the following ones: given a consistent discretization scheme for the SDE system,
\begin{itemize}
\item for any fixed time-step size $\Delta t>0$, there exists a limiting scheme when $\epsilon\to 0$,
\item this limiting scheme is consistent with the limiting equation~\eqref{eq:limit_SDE-av} given by Proposition~\ref{propo:limit-SDE-av} (averaging regime), or~the limiting equation~\eqref{eq:limit_SDE-diff} given by Proposition~\ref{propo:limit-SDE-diff} (diffusion approximation regime).
\end{itemize}
For the SDE considered in this article, consistency is understood in the sense of convergence in distribution. As will be clear below, caution is needed in order to satisfy the second requirement, indeed some standard but naive schemes converge to a limiting scheme which is not consistent with the correct limiting equation. Using such schemes would be dangerous since it could lead to wrong conclusions about the behavior of the SDE system when $\epsilon\to 0$, hence the need to develop simultaneously the theoretical and numerical analysis.

After discussing general properties of AP schemes, we will provide example of such schemes both for the system~\eqref{eq:SDE-av} (averaging regime) and for the system~\eqref{eq:SDE-diff} (diffusion approximation regime)
. We will also study how this scheme applies to the two examples~\eqref{eq:SDE-diff-ex1} and~\eqref{eq:SDE-diff-ex2} described above, and provide a few examples of non AP schemes.

\subsection{Asymptotic Preserving schemes: definition and properties}\label{sec:num-gen}

Let $T\in(0,\infty)$, and let $N\in\mathbb{N}$ and $\Delta t=\frac{T}{N}$ denote the time-step size. Let $\bigl(\Gamma_n\bigr)_{0\le n\le N-1}$ and $\bigl(\gamma_n\bigr)_{0\le n\le N-1}$ be two independent families of independent standard $\R^D$ and $\R$-valued Gaussian random variables. The initial conditions $X_0^\epsilon$ and $m_0^\epsilon$ are assumed to satisfy Assumption~\ref{ass:init}.

On the one hand, a discretization scheme for the SDE~\eqref{eq:SDE-av} is defined as
\begin{equation}\label{eq:scheme-abstract-av}
(X_{n+1}^\epsilon,m_{n+1}^\epsilon)=\Phi_{\Delta t}^{\epsilon}(X_n^\epsilon,m_n^\epsilon,\Gamma_n,\gamma_n),\quad n=0,\ldots,N-1.
\end{equation}
On the other hand, a discretization scheme for the SDE~\eqref{eq:SDE-diff} is defined as
\begin{equation}\label{eq:scheme-abstract-diff}
(X_{n+1}^\epsilon,m_{n+1}^\epsilon)=\Phi_{\Delta t}^{\epsilon}(X_n^\epsilon,m_n^\epsilon,\gamma_n),\quad n=0,\ldots,N-1.
\end{equation}
The presentation is slightly different in the averaging and diffusion approximation regimes. In the remaining of Section~\ref{sec:num-gen}, only the case of schemes of the type~\eqref{eq:scheme-abstract-av} is considered. This means that if one considers the SDE~\eqref{eq:SDE-diff} and the scheme~\eqref{eq:scheme-abstract-diff} (approximation diffusion regime) the variable $\Gamma_n$ needs to be omitted -- this is also the case if $\sigma=0$ in the SDE~\eqref{eq:SDE-av} (averaging regime).

The mapping $\Phi_{\Delta t}^\epsilon$ appearing in the schemes~\eqref{eq:scheme-abstract-av} and~\eqref{eq:scheme-abstract-diff} is referred to as the integrator in the sequel.

Let us first discuss stability issues. Due to the presence of factors $\frac{1}{\epsilon}$ and $\frac{1}{\epsilon^2}$ in the SDE~\eqref{eq:SDE-av} and~\eqref{eq:SDE-diff}, using the standard Euler-Maruyama scheme would impose strong stability conditions, of the type $\Delta t\le \Delta t_0(\epsilon)$ with $\Delta t_0(\epsilon)\to 0$ when $\epsilon\to 0$. In order to study the behavior of the scheme when $\epsilon\to 0$ for any fixed time-step size $\Delta t$, it is necessary to avoid such conditions, and we impose the following assumption (which is generally satisfied for some implicit or implicit-explicit methods).
\begin{hyp}\label{ass:defined_scheme}
The integrator $\Phi_{\Delta t}^\epsilon$ is defined for all $\epsilon\in(0,1]$ and $\Delta t\in(0,\Delta t_0]$, where $\Delta t_0>0$ is independent of $\epsilon$.
\end{hyp}
We are now in position to study the consistency of the scheme. First, it is assumed that for all $\epsilon\in(0,1]$, the scheme~\eqref{eq:scheme-abstract-av} (resp.~\eqref{eq:scheme-abstract-diff}) is consistent with the SDE system~\eqref{eq:SDE-av} (resp.~\eqref{eq:SDE-diff}). When dealing with numerical methods for SDEs, there exist several notions of convergence: in almost sure sense, in probability, in mean-square sense, or in distribution. Since Propositions~\ref{propo:limit-SDE-av} and~\ref{propo:limit-SDE-diff} state that $X^\epsilon$ converges in distribution to $X$ when $\epsilon$, the relevant notion is consistency in the weak sense, related to convergence in distribution.
\begin{hyp}\label{ass:consistency}
For all $\epsilon\in(0,1]$, the numerical scheme~\eqref{eq:scheme-abstract-av} (resp.~\eqref{eq:scheme-abstract-diff}) is consistent in the weak sense with the SDE system~\eqref{eq:SDE-av} (resp.~\eqref{eq:SDE-diff}): for all bounded continuous functions $\varphi:\T^d\times\R$,
\[
\E[\varphi(X_N^\epsilon,m_N^\epsilon)]\xrightarrow[N\to \infty]{} \E[\varphi(X^\epsilon(T),m^\epsilon(T))],
\]
where the time-step size is given by $\Delta t=\frac{T}{N}$, for an arbitrary $T\in(0,\infty)$.
\end{hyp}

Recall that the consistency in the weak sense of the scheme can be verified using the following equivalent criterion, expressed in terms of the integrator and of the infinitesimal generator: for all $\varphi \in \C^2_b(\T^d\times\R)$,
\[
\underset{\Delta t\to 0}\lim~\frac{\E[\varphi(\Phi_{\Delta t}^\epsilon(x,m,\Gamma,\gamma))]-\varphi(x,m)}{\Delta t}=\L^\epsilon \varphi(x,m),
\]
for all $(x,m)\in\T^d\times\R$, where $\Gamma$ and $\gamma$ are two independent standard $\R^D$ and $\R$-valued Gaussian random variables.

The requirements above (Assumptions~\ref{ass:defined_scheme} and~\ref{ass:consistency}) only depend on the behavior of the scheme for fixed $\epsilon\in(0,1]$. We are now in position to study the asymptotic behavior as $\epsilon\to 0$, with fixed time-step size $\Delta t\in(0,\Delta t_0]$. To introduce the notion of Asymptotic Preserving scheme, one first needs to assume the existence of a limiting scheme, as follows.
\begin{hyp}\label{ass:limiting_scheme}
For every $\Delta t\in(0,\Delta t_0]$, there exists a mapping $\Phi_{\Delta t}:\T^d\times\R^2\to \T^d$, such that for every $(x,m)\in\T^d\times\R$, and every bounded continuous function $\varphi:\T^d\to\R$,
\[
\E[\varphi(\Phi_{\Delta t}^{\epsilon}(x,m,\Gamma,\gamma))]\xrightarrow[\epsilon\to 0]{} \E[\varphi(\Phi_{\Delta t}(x,\Gamma,\gamma))]
\]
where $\Gamma$ and $\gamma$ are two independent standard $\R^D$ and $\R$ valued Gaussian random variables.
\end{hyp}
Let $\bigl(X_n\bigr)_{0\le n\le N}$ be defined by
\begin{equation}\label{eq:limitingscheme-recursion}
\begin{aligned}
X_{n+1}&=\Phi_{\Delta t}(X_n,\Gamma_n,\gamma_n),\\
X_0&=x_0=\underset{\epsilon\to 0}\lim~x_0^\epsilon.
\end{aligned}
\end{equation}
where $\bigl(\Gamma_n\bigr)_{0\le n\le N-1}$ and $\bigl(\gamma_n\bigr)_{0\le n\le N-1}$ are two independent families of independent standard $\R^D$ and $\R$ valued Gaussian random variables. By a recursion argument, it is straightforward to check that if Assumptions~~\ref{ass:init} and~\ref{ass:limiting_scheme} are satisfied, then $X_n^\epsilon$ converges in distribution to $X_n$, when $\epsilon\to 0$, for any fixed $\Delta t\in(0,\Delta t_0]$, and $0\le n\le N$.

We are now in position to introduce the notion of Asymptotic Preserving schemes. As for Assumptions~\ref{ass:consistency} and~\ref{ass:limiting_scheme} above, the consistency is understood in the sense of convergence in distribution.
\begin{definition}\label{def:AP}
Let Assumptions~\ref{ass:defined_scheme},~\ref{ass:consistency} and~\ref{ass:limiting_scheme} be satisfied. The scheme~\eqref{eq:scheme-abstract-av} (resp.~\eqref{eq:scheme-abstract-diff}) is said to be Asymptotic Preserving (AP) if the limiting scheme given by Assumption~\ref{ass:limiting_scheme} and~\eqref{eq:limitingscheme-recursion} is consistent, in the weak sense, with the limiting equation given by Proposition~\ref{propo:limit-SDE-av} (resp. Proposition~\ref{propo:limit-SDE-diff}): for every continuous function $\varphi:\T^d\to\R$, one has
\[
\E[\varphi(X_N)]\xrightarrow[N\to\infty]{} \E[\varphi(X(T))],
\]
where $\Delta t=\frac{T}{N}$, with an arbitrary $T\in(0,\infty)$.
\end{definition}

One of the main contributions of this article is the design of AP schemes in the averaging and in the diffusion approximation regimes, see Sections~\ref{sec:num-av} and~\ref{sec:num-diff} respectively.

To conclude this section, Proposition~\ref{propo:AP} and Corollary~\ref{cor:AP} below are general formulations of the AP property in terms of interverting the limits $\epsilon\to 0$ and $\Delta t\to 0$. As explained above, the result is stated only in the averaging regime to simplify the presentation, however the same result holds also in the diffusion approximation regime with straightforward modifications.
\begin{propo}\label{propo:AP}
Let the setting of Definition~\ref{def:AP} be satistied. The following statements are equivalent.
\begin{enumerate}
\item[$(i)$] The scheme~\eqref{eq:scheme-abstract-av} is Asymptotic Preserving.
\item[$(ii)$] For any continuous function $\varphi:\T^d\to\R$, one has
\[
\underset{\Delta t\to 0}\lim~\underset{\epsilon\to 0}\lim~\E[\varphi(X_N^\epsilon)]=\underset{\epsilon\to 0}\lim~\underset{\Delta t\to 0}\lim~\E[\varphi(X_N^\epsilon)],
\]
where $T=N\Delta t$.
\item[$(iii)$] For any $\varphi \in \C^3(\T^d)$, for all $(x,m)\in\T^d\times\R$, one has
\begin{multline*}
\underset{\Delta t\to 0}\lim~\underset{\epsilon\to 0}\lim~\frac{\E[\varphi^\epsilon(\Phi_{\Delta t}^{\epsilon}(x,m,\Gamma,\gamma))]-\varphi(x)}{\Delta t}\\=\underset{\epsilon\to 0}\lim~\underset{\Delta t\to 0}\lim~\frac{\E[\varphi^\epsilon(\Phi_{\Delta t}^{\epsilon}(x,m,\Gamma,\gamma))]-\varphi(x)}{\Delta t},
\end{multline*}
where $\varphi^\epsilon=\varphi+\epsilon\varphi^1$ is the function introduced by the perturbed test function approach (see~\eqref{eq:def_phi-epsilon-av}, Proposition~\ref{propo:limit-SDE-av} or~\eqref{eq:def_phi-epsilon-diff}, Proposition~\ref{propo:limit-SDE-diff}), and $\Gamma$ and $\gamma$ are independent $\R^D$ and $\R$ valued standard Gaussian  random variables.
\end{enumerate}
\end{propo}
Note that using the perturbed test function approach (see Propositions~\ref{propo:limit-SDE-av} and~\ref{propo:limit-SDE-diff}) is the relevant point of view for the statement ~$(iii)$ above.

\begin{proof}[Proof of Proposition~\ref{propo:AP}]
The equivalence of $(i)$ and $(ii)$ is straightforward. Indeed
\begin{align*}
\underset{\Delta t\to 0}\lim~\underset{\epsilon\to 0}\lim~\E[\varphi(X_N^\epsilon)]&=\underset{\Delta t\to 0}\lim~\E[\varphi(X_N)],\\
\underset{\epsilon\to 0}\lim~\underset{\Delta t\to 0}\lim~\E[\varphi(X_N^\epsilon)]&=\underset{\epsilon\to 0}\lim~\E[\varphi(X^\epsilon(T))]=\E[\varphi(X(T))],
\end{align*}
using Assumptions~\ref{ass:consistency} and~\ref{ass:limiting_scheme} and Proposition~\ref{propo:limit-SDE-av}. The two quantities coincide if and only if the limiting scheme is consistent with the limiting equation.

It remains to prove that $(i)$ and $(iii)$ are equivalent. On the one hand, note that
\[
\underset{\Delta t\to 0}\lim~\underset{\epsilon\to 0}\lim~\frac{\E[\varphi^\epsilon(\Phi_{\Delta t}^{\epsilon}(x,m,\Gamma,\gamma))]-\varphi(x)}{\Delta t}=\underset{\Delta t\to 0}\lim~\frac{\E[\varphi(\Phi_{\Delta t}(x,\Gamma,\gamma))]-\varphi(x)}{\Delta t},
\]
using the fact that $\varphi^\epsilon-\varphi={\rm O}(\epsilon)$ and the definition of the limiting scheme from Assumption~\ref{ass:limiting_scheme}.

On the other hand, one has
\[
\underset{\epsilon\to 0}\lim~\underset{\Delta t\to 0}\lim~\frac{\E[\varphi^\epsilon(\Phi_{\Delta t}^{\epsilon}(x,m,\Gamma,\gamma))]-\varphi(x)}{\Delta t}=\underset{\epsilon\to 0}\lim~\L^\epsilon \varphi^\epsilon(x,m)=\L \varphi(x),
\]
using the consistency of the scheme for fixed $\epsilon$ (Assumption~\ref{ass:consistency}), and the property~\eqref{eq:generator_convergence-av}, by construction of the perturbed test function $\varphi^\epsilon$.

Then $(iii)$ is equivalent to having
\[
\underset{\Delta t\to 0}\lim~\frac{\E[\varphi(\Phi_{\Delta t}(x,\gamma))]-\varphi(x)}{\Delta t}=\L \varphi(x),
\]
which means consistency in the weak sense of the limiting scheme with the limiting equation~\eqref{eq:limit_SDE-av}.

This concludes the proof of Proposition~\ref{propo:AP}.
\end{proof}

The following result is a simple criterion to check whether a scheme satisfies the asymptotic preserving property.
\begin{corollary}\label{cor:AP}
Assume that for all $\varphi \in \C^2(\T^d)$, one has
\[
\tilde{\L}\varphi(x)=\underset{\Delta t\to 0}\lim~\frac{\E[\varphi(\Phi_{\Delta t}(x,\gamma))]-\varphi(x)}{\Delta t}
\]
where $\tilde{\L}$ is a second-order differential operator.

Then the scheme is AP if and only if the property stated in $(iii)$ in Proposition~\ref{propo:AP} holds with $\varphi(x)=x_i$ and $\varphi(x)=x_ix_j$, with $1\le i,j\le d$.
\end{corollary}
The proof of Corollary~\ref{cor:AP} is straightforward and is thus omitted.

\subsection{An example of AP scheme in the averaging regime}\label{sec:num-av}

The objective of this section is to propose an example of AP for the SDE model~\eqref{eq:SDE-av}, see Theorem~\ref{th:AP-av}, in the averaging regime. The challenge is to capture the averaged coefficients $\overline{b}$ and $\overline{\sigma}$, given by~\eqref{eq:averagedcoeffs1} and~\eqref{eq:averagedcoeffs2}.

\begin{theo}\label{th:AP-av}
Introduce the following numerical scheme:
\begin{equation}\label{eq:scheme-AP-av}
    \left\lbrace
    \begin{aligned}
        X^\epsilon_{n+1} &= X^\epsilon_n + \Delta t b(X^\epsilon_n,m^\epsilon_{n+1}) + \sqrt{\Delta t} \sigma(X^\epsilon_n,m^\epsilon_{n+1}) \Gamma_n\\
        m^\epsilon_{n+1} &= e^{-\frac{\Delta t}{\epsilon}} m^\epsilon_n + \sqrt{1 - e^{-\frac{2\Delta t}{\epsilon}}} h(X^\epsilon_n) \gamma_n.
    \end{aligned}
    \right.
\end{equation}
This scheme satisfies Assumptions~\ref{ass:defined_scheme},~\ref{ass:consistency} and~\ref{ass:limiting_scheme} and is Asymptotic Preserving in the sense of Definition~\ref{def:AP}. Moreover the limiting scheme is given by
\begin{equation}\label{eq:scheme-limit-av}
    X_{n+1} = X_n + \Delta t b(X_n,h(X_n) \gamma_n) + \sqrt{\Delta t} \sigma(X_n,h(X_n) \gamma_n) \Gamma_n,
\end{equation}
\end{theo}
Let us discuss some properties of the AP scheme~\eqref{eq:scheme-AP-av} and of the limiting scheme~\eqref{eq:scheme-limit-av}. To simplify the discussion, assume that $h(x)=1$. First, assume that $\sigma=0$. Note that even if the limiting equation~\eqref{eq:limit_SDE-av} is a deterministic ordinary differential equation, the scheme~\eqref{eq:scheme-limit-av} is random. However, in that case, the convergence of $X_N$ to $X(T)$ when $\Delta t\to 0$ holds in probability, instead of only in distribution; in that case, the averaging principle result stated in Proposition~\ref{propo:limit-SDE-av} also holds in probability (and even in mean-square sense). The fundamental property to obtain the AP property is that the random quantity appearing in the limiting scheme~\eqref{eq:scheme-limit-av} satisfies the property
\begin{equation}\label{eq:conditioning-av}
    \E[b(X_n,h(X_n)\gamma_n)|X_n]=\overline{b}(X_n).
\end{equation}
In the AP scheme~\eqref{eq:scheme-AP-av}, the fast component is discretized exactly in distribution (when $h(x)=1$): for all $n\ge 0$, the Gaussian random variables $m_n^\epsilon$ and $m^\epsilon(n\Delta t)$ are equal in distribution. The fundamental property written above cannot be satisfied if one uses for instance the implicit Euler scheme to discretize the fast component: the scheme defined by
\begin{equation}\label{eq:scheme-nonAP-av}
    \left\lbrace
    \begin{aligned}
        X^\epsilon_{n+1} &= X^\epsilon_n + \Delta t b(X^\epsilon_n,m^\epsilon_{n+1})\\
        m^\epsilon_{n+1} &= m_n^\epsilon-\frac{\Delta t}{\epsilon}m^\epsilon_{n+1}+\sqrt{2\frac{\Delta t}{\epsilon}}\gamma_n,
    \end{aligned}
    \right.
\end{equation}
is not asymptotic preserving, since the associated limiting scheme is
\[
X_{n+1}=X_n+\Delta tb(X_n,0).
\]
using the identity
\[
m_{n+1}^\epsilon=\frac{1}{1+\frac{\Delta t}{\epsilon}}m_n^\epsilon+\frac{\sqrt{2\frac{\Delta t}{\epsilon}}}{1+\frac{\Delta t}{\epsilon}}\gamma_n\underset{\epsilon\to 0}\to 0,
\]
to pass to the limit.

Second, assume that $\sigma$ is not equal to $0$. Then the convergence of $X_n^\epsilon$ to $X_n$ only holds in distribution in general. It does not hold in mean-square sense in the following case: assume that $d=1$, and that $b(x)=0$ and $\sigma(x,m)=m$ (in that example, the convergence in Proposition~\ref{propo:limit-SDE-av} also does not hold in the mean-square sense). Assume also for simplicity that $x_0^\epsilon=0$, and that $m_0^\epsilon=m_0\sim\mathcal{N}(0,1)$ (and is independent of the Wiener processes $\beta$ and $B$). Then one has $\overline{\sigma}(x)=1$, thus
\[
X_{n}=\sum_{k=0}^{n-1}\sqrt{\Delta t}\Gamma_k\quad,\quad
X_{n}^{\epsilon}=\sum_{k=0}^{n-1}\sqrt{\Delta t}m_{k+1}^\epsilon\Gamma_k,
\]
and one obtains
\[
\E|X_n^\epsilon-X_n|^2=\Delta t\sum_{k=0}^{n-1}\E|m_{k+1}^\epsilon-1|^2=n\Delta t\E|m_0-1|^2,
\]
and the right-hand side does not depend on $\epsilon$. It is thus natural to consider convergence in distribution in the notion of asymptotic preserving schemes for SDEs.

Finally, note also that, as above, the scheme
\begin{equation*}
    \left\lbrace
    \begin{aligned}
        X^\epsilon_{n+1} &= X^\epsilon_n + \sqrt{\Delta t}\sigma(X_n^\epsilon,m_{n+1}^\epsilon)\Gamma_n\\
        m^\epsilon_{n+1} &= m_n^\epsilon-\frac{\Delta t}{\epsilon}m^\epsilon_{n+1}+\sqrt{2\frac{\Delta t}{\epsilon}}\gamma_n,
    \end{aligned}
    \right.
\end{equation*}
is not asymptotic preserving, since the associated limiting scheme is
\[
X_{n+1}=X_n+\sqrt{\Delta t}\sigma(X_n,0)\Gamma_n.
\]

We are now in position to prove Theorem~\ref{th:AP-av}.
\begin{proof}[Proof of Theorem~\ref{th:AP-av}]
It is straightforward to check that Assumption~\ref{ass:defined_scheme} is satisfied. Let us prove that Assumption~\ref{ass:limiting_scheme} holds. We have
\begin{align*}
    \Phi_{\Delta t}^{\epsilon}(x,m,\Gamma,\gamma) &= x + \Delta t b(x,m') + \sqrt{\Delta t} \sigma(x,m') \Gamma\\
    \Phi_{\Delta t}(x,\Gamma,\gamma) &= x + \Delta t b(x,h(x)\gamma) + \sqrt{\Delta t} \sigma(x,h(x)\gamma) \Gamma,
\end{align*}
with $m'= e^{-\frac{\Delta t}{\epsilon}} m + \sqrt{1 - e^{-\frac{2\Delta t}{\epsilon}}} h(x) \gamma$. When $\epsilon \to 0$, $m'$ converges almost surely to $h(x) \gamma$, thus $\Phi_{\Delta t}^{\epsilon}(x,m,\Gamma,\gamma)$ converges in distribution to $\Phi_{\Delta t}(x,\Gamma,\gamma)$, and Assumption~\ref{ass:limiting_scheme} is satisfied.

It remains to prove that the scheme satisfies Assumption~\ref{ass:consistency} and is asymptotic preserving in the sense of Definition~\ref{def:AP}, namely that the schemes~\eqref{eq:scheme-AP-av} and~\eqref{eq:scheme-limit-av} are consistent (in the weak sense), with~\eqref{eq:SDE-av} and~\eqref{eq:limit_SDE-av} respectively.

Let $\epsilon > 0$ be fixed.
Since $h$ is bounded, it is straightforward to check that
\begin{equation} \label{eq:moment_m-eps-n}
    \sup_{\epsilon\in(0,1]} \sup_{n\ge 0} \frac{\E[\abs{m^\epsilon_n}^2]}{1 + \abs{m^\epsilon_0}^2} < \infty.
\end{equation}
This estimate will prove useful to prove Lemma~\ref{lem4} of Theorem~\ref{th:UA-av}. Since $b$ and $\sigma$ are also bounded, we get, in $\mathbb L^1(\Omega)$, when $\Delta t \to 0$
\begin{align*}
    m^\epsilon_{n+1} - m^\epsilon_n &= - \frac{\Delta t}{\epsilon} m^\epsilon_n + \sqrt{\frac{2\Delta t}{\epsilon}} h(X^\epsilon_n) \gamma_n + {\rm o}(\Delta t),\\
    X^\epsilon_{n+1} - X^\epsilon_n &= \Delta t b(X_n^\epsilon) + \sqrt{\Delta t} \sigma(X_n^\epsilon) \Gamma_n + {\rm o}(\Delta t).
\end{align*}
Thus, using that $\Gamma_n$, $\gamma_n$ and $X^\epsilon_n$ are independent, we get the second order Taylor expansion of $\varphi \in \mathcal C^2_b(\T^d \times \R)$,
\begin{align*}
    \E[\varphi(&X_{n+1}^\epsilon,m_{n+1}^\epsilon)] - \E[\varphi(X_n^\epsilon,m_n^\epsilon)]\\
    &= \Delta t \E[b(X^\epsilon_n,m^\epsilon_{n+1}) \cdot \nabla_x \varphi(X^\epsilon_n,m^\epsilon_n)]+ \frac{1}{2} \Delta t \E[\sigma \sigma^* (X^\epsilon_n,m^\epsilon_{n+1}) : \nabla_x^2 \varphi(X^\epsilon_n,m^\epsilon_n)]\\
    &\phantom{=} - \frac{\Delta t}{\epsilon} \E[m^\epsilon_n \partial_m \varphi(X^\epsilon_n,m^\epsilon_n)]+ \frac{\Delta t}{\epsilon} \E[h(X^\epsilon_n)^2 \partial_m^2 \varphi(X^\epsilon_n,m^\epsilon_n)] + {\rm o}(\Delta t)\\
    &= \Delta t \E[\L^\epsilon \varphi(X^\epsilon_n,m^\epsilon_n)] + {\rm o}(\Delta t).
\end{align*}
From there, it is straightforward to check that Assumption~\ref{ass:consistency} is satisfied.

Similarly, to prove the consistency of the limiting scheme~\eqref{eq:scheme-limit-av} with~\eqref{eq:limit_SDE-av}, for $\varphi \in \mathcal C^2(\T^d)$, when $\Delta t \to 0$, observe that one has
\begin{multline*}
    \E[\varphi(X_{n+1})] - \E[\varphi(X_n)] = \Delta t \E[b(X_n,h(X_n) \gamma_n) \cdot \nabla_x \varphi(X_n)]\\
    + \frac{1}{2} \Delta t \E[\sigma \sigma^* (X_n,h(X_n) \gamma_n) : \nabla_x^2 \varphi(X_n)] + {\rm o}(\Delta t).
\end{multline*}
The key argument of this proof is the following: by conditioning with respect to $X_n$ and the definitions~\eqref{eq:averagedcoeffs1}--\eqref{eq:averagedcoeffs2} of the averaged coefficients, using the fundamental property~\eqref{eq:conditioning-av} for $\overline b$ and $\overline \sigma \, \overline \sigma^* = \overline{\sigma \sigma^*}$, yields
\begin{equation*}
    \E[\varphi(X_{n+1})] = \E[\varphi(X_n)] + \Delta t \E[\L \varphi(X_n)] + {\rm o}(\Delta t).
\end{equation*}
The limiting scheme is thus consistent with the limiting equation. This concludes the proof of Theorem~\ref{th:AP-av}.
\end{proof}

Beyond the asymptotic preserving property, it is possible to obtain error estimate, and to prove that the scheme~\eqref{eq:scheme-AP-av} given in Theorem~\ref{th:AP-av} is uniformly accurate (in distribution).
\begin{theo}\label{th:UA-av}
Let Assumptions~\ref{ass:init},~\ref{ass:coeffs-av} and~\ref{ass:av} be satisfied. For any $T\in(0,\infty)$ and any function $\varphi:\T^d\to\R$ of class $\C^4$, there exists $C(T,\varphi)\in(0,\infty)$ such that for all $\Delta t\in(0,\Delta t_0]$ and $\epsilon\in(0,1]$ one has
\begin{equation}\label{eq:UA-1}
\big|\E[\varphi(X_N^\epsilon)-\E[\varphi(X^\epsilon(T))]\big|\le C(T,\varphi)\min\Bigl(\frac{\Delta t}{\epsilon},\Delta t+\epsilon\Bigr),
\end{equation}
and the scheme~\eqref{eq:scheme-AP-av} is uniformly accurate with the following error estimate: for all $\Delta t\in(0,\Delta t_0]$, one has
\begin{equation}\label{eq:UA-2}
\underset{\epsilon\in(0,1]}\sup~\big|\E[\varphi(X_N^\epsilon)-\E[\varphi(X^\epsilon(T))]\big|\le C(T,\varphi)\sqrt{\Delta t}.
\end{equation}
\end{theo}
The error estimate~\eqref{eq:UA-2} implies that the error $\big|\E[\varphi(X_N^\epsilon)-\E[\varphi(X^\epsilon(T))]\big|$ goes to $0$ when $\Delta t\to 0$ uniformly with respect to $\epsilon\to 0$. Note that~\eqref{eq:UA-2} is a straightforward consequence of~\eqref{eq:UA-1}, considering the cases $\sqrt{\Delta t}\le \epsilon$ and $\epsilon\le \sqrt{\Delta t}$ separately. This argument implies a reduction in the order of convergence appearing in~\eqref{eq:UA-2}: it is equal to $\frac12$ whereas for fixed $\epsilon>0$ (in~\eqref{eq:UA-1}) or when $\epsilon=0$ the order of convergence is equal to $1$.

The proof of Theorem~\ref{th:UA-av} is long, technical and requires several auxiliary results, it is thus postponed to Section~\ref{sec:error}.

\subsection{An example of AP scheme in the diffusion approximation regime}\label{sec:num-diff}

The objective of this section is to propose an example of AP scheme for the SDE model~\eqref{eq:SDE-diff}, see Theorem~\ref{th:AP-diff}, in the diffusion approximation regime. The challenge is to let the limiting scheme capture the additional drift term appearing in the limiting equation~\eqref{eq:limit_SDE-diff} when $\sigma$ or $f$ is not constant.
\begin{theo}\label{th:AP-diff}
Let $\theta \in [\frac 1 2, 1]$.
Introduce the following numerical scheme:
\begin{equation}\label{eq:scheme-AP-diff-general}
    \left\lbrace
    \begin{aligned}
        \hat{m}_{n+1}^\epsilon &= m_n^\epsilon - \frac{\Delta t f(X_n^\epsilon) \hat m^\epsilon_{n+\theta}}{\epsilon^2} + \frac{\Delta t f(X_n^\epsilon) g(X_n^\epsilon)}{\epsilon} + \frac{f(X_n^\epsilon) h(X_n^\epsilon) \sqrt{\Delta t} \gamma_n}{\epsilon},\\
        \hat{X}_{n+1}^\epsilon &= X_n^\epsilon + \Delta t b(X_n^\epsilon) + \sigma(X_n^\epsilon) \frac{\Delta t \hat m^\epsilon_{n+\theta}}{\epsilon},\\
        m_{n+1}^\epsilon &= m_n^\epsilon - \frac{\Delta t f(\hat{X}_{n+1}^\epsilon) m^\epsilon_{n+\theta}}{\epsilon^2} + \frac{\Delta t f(\hat{X}_{n+1}^\epsilon) g(X_n^\epsilon)}{\epsilon} + \frac{f(X_n^\epsilon) h(X_n^\epsilon) \sqrt{\Delta t} \gamma_n}{\epsilon},\\
        Y_{n+1}^\epsilon &= X_n^\epsilon + \Delta t b(X_n^\epsilon) + \sigma(X_n^\epsilon) \frac{\Delta t m^\epsilon_{n+\theta}}{\epsilon},\\
        X_{n+1}^\epsilon &= X_n^\epsilon + \Delta t b(X_n^\epsilon) + \frac{\sigma(X_n^\epsilon) + \sigma(Y_{n+1}^\epsilon)}{2} \frac{\Delta t}{\epsilon} \frac{\hat m^\epsilon_{n+\theta} + m^\epsilon_{n+\theta}}{2},
    \end{aligned}
    \right.
\end{equation}
where
\begin{gather*}
    \hat m^\epsilon_{n+\theta} = (1-\theta) m^\epsilon_n + \theta \hat m^\epsilon_{n+1}\\
    m^\epsilon_{n+\theta} = (1-\theta) m^\epsilon_n + \theta m^\epsilon_{n+1}.
\end{gather*}
This scheme satisfies Assumptions~\ref{ass:defined_scheme},~\ref{ass:consistency} and~\ref{ass:limiting_scheme} and is Asymptotic Preserving in the sense of Definition~\ref{def:AP}. Moreover the limiting scheme is given by
\begin{equation}\label{eq:scheme-limit-diff-general}
    \left\lbrace
    \begin{aligned}
        \hat{X}_{n+1} &= X_n + \Delta t \left( b(X_n) + g(X_n) \sigma(X_n) \right) + \sigma(X_n) h(X_n) \sqrt{\Delta t} \gamma_n,\\
        Y_{n+1} &= X_n + \Delta t \bigl( b(X_n) + g(X_n) \sigma(X_n) \bigr) + \sigma(X_n) h(X_n) \frac{f(X_n)}{f(\hat{X}_{n+1})} \sqrt{\Delta t} \gamma_n,\\
        X_{n+1} &= X_n + \Delta t \left(b(X_n) + g(X_n) \frac{\sigma(X_n) + \sigma(Y_{n+1})}{2}\right)\\
        &\quad + \frac{\sigma(X_n) + \sigma(Y_{n+1})}{2} \frac{1 + \frac{f(X_n)}{f(\hat{X}_{n+1})}}{2} h(X_n) \sqrt{\Delta t} \gamma_n.
    \end{aligned}
    \right.
\end{equation}
\end{theo}
The design of the scheme~\ref{eq:scheme-AP-diff-general} is based on a carefully chosen prediction-correction procedure. The limiting scheme~\eqref{eq:scheme-limit-diff-general} then also contains prediction steps which are the key elements to satisfy the consistency with the limiting SDE~\eqref{eq:limit_SDE-diff}. The choice of the prediction-correction procedure is made clearer looking at the two examples~\eqref{eq:SDE-diff-ex1} and~\eqref{eq:SDE-diff-ex2}, see below Corollaries~\ref{cor:diff-ex1} and~\ref{cor:diff-ex2} respectively. The prediction-correction procedure is crucial to obtain the AP property for the scheme: the following simpler scheme (with $\theta=1$ to simplify the presentation)
\begin{equation} \label{eq:scheme-nonAP-diff}
\left\lbrace
\begin{aligned}
X_{n+1}^\epsilon&=X_n+\Delta tb(X_n^\epsilon)+\sigma(X_n^\epsilon)\frac{\Delta t}{\epsilon}m_{n+1}^\epsilon,\\
m_{n+1}^\epsilon&=m_n^\epsilon-\frac{f(X_n^\epsilon)\Delta t}{\epsilon^2}m_{n+1}^\epsilon+\frac{f(X_n^\epsilon)g(X_n^\epsilon)\Delta t}{\epsilon}+\frac{f(X_n^\epsilon)h(X_n^\epsilon)\sqrt{\Delta t}}{\epsilon}\gamma_n,
\end{aligned}
\right.
\end{equation}
is not asymptotic preserving, since the associated limiting scheme (see the proof of Theorem~\ref{th:AP-diff} for the derivation of the limiting scheme) is
\[
X_{n+1}=X_n+\Delta t\bigl(b(X_n)+g(X_n)\sigma(X_n)\bigr)+h(X_n)\sigma(X_n)\sqrt{\Delta t}\gamma_n.
\]
This limiting scheme is consistent with the SDE $dX_t=\bigl(b(X_t)+g(X_t)\sigma(X_t)\bigr)dt+h(X_t)\sigma(X_t)dW_t$, which differs in general -- when $\sigma$ or $f$ is non constant -- from the correct limiting equation~\eqref{eq:limit_SDE-diff}.

Observe that in the AP scheme~\eqref{eq:scheme-AP-diff-general} the fast component $m^\epsilon$ is discretized using the $\theta$-method. Choosing $\theta\in[\frac12,1]$ ensures the mean-square stability of the scheme (Assumption~\ref{ass:defined_scheme}), uniformly with respect to $\epsilon$. Note that the same quantity $(1-\theta)m_n^\epsilon + \theta\hat{m}_{n+1}^\epsilon$ appears in the expressions of $\hat{m}_{n+1}^\epsilon$ and $\hat{X}_{n+1}^{\epsilon}$ in~\eqref{eq:scheme-AP-diff-general}. Similarly, the same quantity $(1-\theta)m_n^\epsilon + \theta{m}_{n+1}^\epsilon$ appears in the expressions of ${m}_{n+1}^\epsilon$ and $Y_{n+1}^{\epsilon}$ in~\eqref{eq:scheme-AP-diff-general}: this highlights the fact that in order to get a limiting scheme, it is fundamental to choose the quadrature rules in this consistent way.

\begin{rem}
There would be no loss of generality to assume that $b=0$. Another example of AP scheme would be obtained in the case $b\neq 0$, using a splitting technique: combining the scheme~\eqref{eq:scheme-AP-diff-general} with $b=0$, with a standard explicit Euler scheme to treat the contribution of $b$. Writing the expression of the resulting scheme is left to the reader.
\end{rem}

\begin{proof}[Proof of Theorem~\ref{th:AP-diff}]
It is straightforward to check that Assumption~\ref{ass:defined_scheme} is satisfied.

Let us prove that Assumption~\ref{ass:limiting_scheme} holds, namely that~\eqref{eq:scheme-AP-diff-general} converges to~\eqref{eq:scheme-limit-diff-general} when $\epsilon \to 0$. Note that for fixed $\Delta t>0$ and $0\le n\le N$, one has
\begin{equation} \label{eq:moment-scheme-AP-general}
    \sup_{\epsilon > 0} \E[|m_n^\epsilon| + |\hat m_n^\epsilon|] < +\infty.
\end{equation}
This is proved by a straightforward recursion argument. As a consequence, one obtains convergence of the quantity,
\begin{equation*}
        \frac{\Delta t\hat m^\epsilon_{n+\theta}}{\epsilon}  = \Delta t g(X_n^\epsilon) + h(X_n^\epsilon) \sqrt{\Delta t}\gamma_n - \frac{\epsilon}{f(X_n^\epsilon)} \left( \hat m_{n+1}^\epsilon - m_n^\epsilon \right).
\end{equation*}
Thus one has $Y_{n+1}^\epsilon \xrightarrow[\epsilon \to 0]{} Y_{n+1}$ and $\hat X_{n+1}^\epsilon \xrightarrow[\epsilon \to 0]{} \hat X_{n+1}$. Similarly, one obtains the convergence of $\frac{\Delta t m^\epsilon_{n+\theta}}{\epsilon}$, which yields $Y_{n+1}^\epsilon \xrightarrow[\epsilon \to 0]{} Y_{n+1}$ and $X_{n+1}^\epsilon \xrightarrow[\epsilon \to 0]{} X_{n+1}$.

It remains to prove that the scheme satisfies Assumption~\ref{ass:consistency} and is asymptotic preserving in the sense of Definition~\ref{def:AP}, namely that the schemes~\eqref{eq:scheme-AP-diff-general} and~\eqref{eq:scheme-limit-diff-general} are consistent (in the weak sense), with~\eqref{eq:SDE-diff} and~\eqref{eq:limit_SDE-diff} respectively.

On the one hand, let $\epsilon > 0$ be fixed. To prove that~\eqref{eq:scheme-AP-diff-general} is consistent with~\eqref{eq:SDE-diff}, it is sufficient to prove that, for $\varphi \in \C^2_b(\T^d \times \R)$, when $\Delta t \to 0$,
\begin{equation} \label{eq:consistency-scheme-AP-diff}
    \E[\varphi(X_{n+1}^\epsilon,m_{n+1}^\epsilon)] = \E[\varphi(X_n^\epsilon,m_n^\epsilon)] + \Delta t \E[\L^\epsilon \varphi(X_n^\epsilon,m_n^\epsilon)] + {\rm o}(\Delta t).
\end{equation}
It is straightforward to check that, in $\mathbb L^1(\Omega)$,
\begin{equation*}
    \left( \hat m_{n+1}^\epsilon, \hat X_{n+1}^\epsilon, m_{n+1}^\epsilon, Y_{n+1}^\epsilon \right) = \left( \hat m_n^\epsilon, X_n^\epsilon, m_n^\epsilon, X_n^\epsilon \right) + {\rm o}(1),
\end{equation*}
hence
\begin{align*}
    m_{n+1}^\epsilon &= m_n^\epsilon + \sqrt{\Delta t} \frac{f(X_n^\epsilon) h(X_n^\epsilon) \gamma_n}{\epsilon} + \Delta t \left( - \frac{f(X_n^\epsilon) m_n^\epsilon}{\epsilon^2} + \frac{f(X_n^\epsilon) g(X_n^\epsilon)}{\epsilon} \right) + {\rm o}(\Delta t),\\
    X_{n+1}^\epsilon &= X_n^\epsilon + \Delta t \left( b(X_n^\epsilon) + \frac{\sigma(X_n^\epsilon) m_n^\epsilon}{\epsilon} \right) + {\rm o}(\Delta t).
\end{align*}
Since $\gamma_n \sim \mathcal N(0,1)$ and the random variables $\gamma_n$ and $X_n$ are independent, one obtains~\eqref{eq:consistency-scheme-AP-diff}.

On the other hand, it remains to prove that the limiting scheme~\eqref{eq:scheme-limit-diff-general} is consistent with~\eqref{eq:limit_SDE-diff}, {\it i.e.} that, for $\varphi \in \C^2(\T^d)$, when $\Delta t \to 0$,
\begin{equation} \label{eq:consistency-scheme-AP-limit}
    \E[\varphi(X_{n+1})] = \E[\varphi(X_n)] + \Delta t \E[\L \varphi(X_n)] + {\rm o}(\Delta t).
\end{equation}
To simplify the presentation, for any function $\psi$, the following notation is used below:
\begin{equation*}
    \psi_n \doteq \psi(X_n), \quad \nabla \psi_n \doteq \nabla \psi(X_n)
\end{equation*}

The key argument of this proof is the analysis of the asymptotic behavior of the quantity $\frac{\sigma(X_n) + \sigma(Y_{n+1})}{2} \frac{1 + \frac{f(X_n)}{f(\hat{X}_{n+1})}}{2}$, which appears in the scheme in order to capture the drift terms in the limiting equation~\eqref{eq:limit_SDE-diff}.

First, performing expansions at order $\sqrt{\Delta t}$ for $\hat{X}_{n+1}$ and $Y_{n+1}$ yields
\begin{equation*}
    \hat{X}_{n+1} = X_n + \sqrt{\Delta t} h_n \gamma_n \sigma_n + {\rm o}(\sqrt{\Delta t}), \quad Y_{n+1} = X_n + \sqrt{\Delta t} h_n \gamma_n \sigma_n + {\rm o}(\sqrt{\Delta t}).
\end{equation*}
Second, writing $\frac{\sigma_n + \sigma(Y_{n+1})}{2} = \sigma_n + \frac{\sigma(Y_{n+1}) - \sigma_n}{2}$ and $\frac{f_n}{f(\hat{X}_{n+1})} = \frac{1}{1 + \frac{f(\hat{X}_{n+1}) - f_n}{f_n}}$, one obtains the following expansion at order $\sqrt{\Delta t}$ for the quantity:
\begin{align*}
    &\frac{\sigma_n + \sigma(Y_{n+1})}{2} \frac{1 + \frac{f_n}{f(\hat{X}_{n+1})}}{2}\\
    &= \left( \sigma_n + \frac 1 2 \sqrt{\Delta t} h_n \gamma_n (\sigma_n \cdot \nabla) \sigma_n + {\rm o}(\sqrt{\Delta t}) \right)\left( 1 - \frac 1 {2 f_n} \sqrt{\Delta t} h_n \gamma_n \sigma_n \cdot \nabla f_n+ {\rm o}(\sqrt{\Delta t}) \right)\\
    &= \sigma_n + \sqrt{\Delta t} \left( - \frac 1 {2 f_n} h_n \gamma_n \sigma_n \cdot \nabla f_n \sigma_n + \frac 1 2 h_n \gamma_n (\sigma_n \cdot \nabla) \sigma_n \right)\\
    &\phantom{=} + {\rm o}(\sqrt{\Delta t}).
\end{align*}
Finally, one obtains the following asymptotic expansion of $X_{n+1}$
\begin{multline*}
    X_{n+1} = X_n + \sqrt{\Delta t} h_n \gamma_n \sigma_n\\+ \Delta t \left( b_n + g_n \sigma_n - \frac 1 {2 f_n} h_n^2 \gamma_n^2 \sigma_n \cdot \nabla f_n \sigma_n + \frac 1 2 h_n^2 \gamma_n^2 (\sigma_n \cdot \nabla) \sigma_n \right) + {\rm o}(\Delta t).
\end{multline*}
Since $\gamma_n$ is centered and $\gamma_n$ and $X_n$ are independent random variables, one obtains the first order expansion~\eqref{eq:consistency-scheme-AP-limit}.

This concludes the proof of Theorem~\ref{th:AP-diff}.
\end{proof}

The proposed AP scheme given by Theorem~\ref{th:AP-diff} can be simplified when it is applied to one of the two examples of SDE models introduced in Section~\ref{sec:models-diff-ex}. These two examples are employed in the numerical experiments below. To simplify the presentation, we only consider the case $\theta=1$.

\begin{corollary}\label{cor:diff-ex1}
Consider the SDE~\eqref{eq:SDE-diff-ex1}. The AP scheme~\eqref{eq:scheme-AP-diff-general} given by Theorem~\ref{th:AP-diff} is written as follows:
\begin{equation}\label{eq:scheme-AP-diff-ex1}
\left\lbrace
\begin{aligned}
m_{n+1}^\epsilon&=m_n^\epsilon-\frac{\Delta t}{\epsilon^2}m_{n+1}^\epsilon+\frac{\sqrt{\Delta t}}{\epsilon}\gamma_n,\\
Y_{n+1}^\epsilon&=X_n^\epsilon+\sigma(X_n^\epsilon)\frac{\Delta t m_{n+1}^\epsilon}{\epsilon},\\
X_{n+1}^\epsilon&=X_n^\epsilon+\frac{\sigma(X_n^\epsilon)+\sigma(Y_{n+1}^\epsilon)}{2}\frac{\Delta tm_{n+1}^\epsilon}{\epsilon}.
\end{aligned}
\right.
\end{equation}
The scheme~\eqref{eq:scheme-AP-diff-ex1} is Asymptotic Preserving, and the limiting scheme~\eqref{eq:scheme-limit-diff-general} is written as
\begin{equation}\label{eq:scheme-limit-diff-ex1}
\left\lbrace
\begin{aligned}
Y_{n+1}&=X_n+\sqrt{\Delta t}\sigma(X_n)\gamma_n,\\
X_{n+1}&=X_n+\sqrt{\Delta t}\frac{\sigma(X_n)+\sigma(Y_{n+1})}{2}\gamma_n.
\end{aligned}
\right.
\end{equation}
The limiting scheme~\eqref{eq:scheme-limit-diff-ex1} is consistent with the limiting SDE~\eqref{eq:SDE-limit-diff-ex1}.
\end{corollary}
The prediction-correction procedure appearing in the limiting scheme~\eqref{eq:scheme-limit-diff-ex1} allows to recover the Stratonovich interpretation of the noise in the limiting SDE~\eqref{eq:SDE-limit-diff-ex1}.

\begin{corollary}\label{cor:diff-ex2}
Consider the SDE~\eqref{eq:SDE-diff-ex2}. The AP scheme~\eqref{eq:scheme-AP-diff-general} given by Theorem~\ref{th:AP-diff} is written as follows:
\begin{equation}\label{eq:scheme-AP-diff-ex2}
\left\lbrace
\begin{aligned}
\hat{m}_{n+1}^\epsilon&=m_n^\epsilon-\frac{\Delta tf(X_n^\epsilon)}{\epsilon^2}\hat{m}_{n+1}^\epsilon+\frac{\Delta tf(X_n^\epsilon)g(X_n^\epsilon)}{\epsilon}+\frac{f(X_n^\epsilon)h(X_n^\epsilon)\sqrt{\Delta t}\gamma_n}{\epsilon}\\
\hat{X}_{n+1}^\epsilon&=X_n+\frac{\Delta t\hat{m}_{n+1}^\epsilon}{\epsilon}\\
m_{n+1}^\epsilon&=m_n^\epsilon-\frac{\Delta tf(\hat{X}_{n+1}^\epsilon)}{\epsilon^2}{m}_{n+1}^\epsilon+\frac{\Delta tf(\hat{X}_{n+1}^\epsilon)g(X_n^\epsilon)}{\epsilon}+\frac{f(X_n^\epsilon)h(X_n^\epsilon)\sqrt{\Delta t}\gamma_n}{\epsilon}\\
X_{n+1}^\epsilon&=X_n^\epsilon+\frac{\Delta t}{\epsilon}\frac{\hat{m}_{n+1}^\epsilon+m_{n+1}^\epsilon}{2}
\end{aligned}
\right.
\end{equation}
The scheme~\eqref{eq:scheme-AP-diff-ex2} is Asymptotic Preserving, and the limiting scheme~\eqref{eq:scheme-limit-diff-general} is written as
\begin{equation}\label{eq:scheme-limit-diff-ex2}
\left\lbrace
\begin{aligned}
\hat{X}_{n+1}&=X_n+\Delta tg(X_n)+h(X_n)\sqrt{\Delta t}\gamma_n,\\
X_{n+1}&=X_n+\Delta tg(X_n)+\frac{1+\frac{f(X_n)}{f(\hat{X}_{n+1})}}{2}h(X_n)\sqrt{\Delta t}\gamma_n.
\end{aligned}
\right.
\end{equation}
The limiting scheme~\eqref{eq:scheme-limit-diff-ex2} is consistent with the limiting SDE~\eqref{eq:SDE-limit-diff-ex2}.
\end{corollary}
The prediction-correction procedure appearing in the limiting scheme~\eqref{eq:scheme-limit-diff-ex1} allows to recover the noise-induced drift term appearing in the limiting SDE~\eqref{eq:SDE-limit-diff-ex2}.

\begin{rem}\label{rem:ex1bis}
Consider the first example~\eqref{eq:SDE-diff-ex1}, with $\sigma(x)=x$ and assume that $x$ takes values in the real line $\R$ (instead of the torus $\T$). The following scheme
\begin{equation}\label{eq:scheme-AP-ex1bis}
\left\lbrace
\begin{aligned}
X_{n+1}^\epsilon&=X_n^\epsilon \exp\bigl(\frac{\Delta t}{\epsilon}[(1-\theta)m_n^\epsilon+\theta m_{n+1}^\epsilon]\bigr)\\
m_{n+1}^\epsilon&=m_n^\epsilon-\frac{\Delta t}{\epsilon^2}\bigl[(1-\theta)m_n^\epsilon+\theta m_{n+1}^\epsilon\bigr]+\frac{\sqrt{\Delta t}}{\epsilon}\gamma_n,
\end{aligned}
\right.
\end{equation}
where $\theta\in[\frac12,1]$, is another example of AP scheme for this problem. The limiting scheme is given by
\begin{equation}\label{eq:scheme-limit-ex1bis}
X_{n+1}=X_n\exp(\sqrt{\Delta t}\gamma_n)
\end{equation}
which is consistent with the limiting equation $dX_t=X_t\circ dW_t$.

However, the construction is more subtle if the fast component is discretized using an exponential method: let
\[
m_{n+1}^\epsilon=e^{-\frac{\Delta t}{\epsilon^2}}m_n^\epsilon+\sqrt{\frac{1-e^{-2\frac{\Delta t}{\epsilon^2}}}{2}}\gamma_n,
\]
then $m_{n+1}^\epsilon\underset{\epsilon\to 0}\to \frac{1}{\sqrt{2}}\gamma_n$, and defining
\[
X_{n+1}^\epsilon=X_n^\epsilon \exp\bigl(\frac{\Delta t}{\epsilon}m_{n+1}^\epsilon\bigr)
\]
does not provide an AP scheme, since there exists no limiting scheme when $\epsilon\to 0$.

Inspired by the identity $\frac{1}{\epsilon}\int_0^t m^\epsilon(s) ds=\epsilon(m_0^\epsilon-m^\epsilon(t))+\beta_t$ for all $t\ge 0$, one may set
\[
X_{n+1}^\epsilon=X_n^\epsilon \exp\bigl(\epsilon(m_{n}^\epsilon-m_{n+1}^\epsilon)+\sqrt{\Delta t}\gamma_n\bigr),
\]
so that one has
\[
X_N^\epsilon=x_0^\epsilon\exp\bigl(\epsilon(m_0^\epsilon-m_N^\epsilon)+\sqrt{\Delta t}\sum_{n=0}^{N-1}\gamma_n\bigr).
\]
When $\epsilon\to 0$, one obtains the limiting scheme~\eqref{eq:scheme-limit-ex1bis} again. However, the generalization of this construction to the case $\sigma(x)\neq x$ is not straightforward, whereas the scheme proposed in Theorem~\ref{th:AP-diff} applies directly to the general case.
\end{rem}

\section{Numerical experiments}\label{sec:exp}

To simplify the discussion, the dimension $d$ is set equal to $1$.

\subsection{Illustration in the averaging regime}

The objective of this section is to illustrate qualitatively the superiority of the AP scheme~\eqref{eq:scheme-AP-av} proposed in Section~\ref{sec:num-av}, when the parameter $\epsilon$ is small, compared with the use of crude integrators which are not AP. In particular, the numerical experiments below confirm that the limiting scheme~\eqref{eq:scheme-limit-av} is consistent with the limiting equation~\eqref{eq:limit_SDE-av}.

We consider the equation~\eqref{eq:SDE-av} with a drift given by $b(x,m) = \cos(2 \pi x) e^{-\frac{m^2}{2}}$, and diffusion coefficient $\sigma(x,m)=0$. Let $T=1$, $x_0^\epsilon=1$ and $m_0^\epsilon=0$.

Recall that the AP scheme is given by~\eqref{eq:scheme-AP-av}, the limiting scheme is given by~\eqref{eq:scheme-limit-av} and the limiting equation is given by~\eqref{eq:limit_SDE-av}, with $\overline b(x) = \frac{1}{\sqrt{2}} \cos(2 \pi x)$ and $\overline \sigma = 0$. Let us define $X_n^{\rm ref}$ using the standard Euler scheme applied to this limiting equation:
\begin{equation}\label{eq:scheme-ref-av}
    X_{n+1}^{\rm ref} = X_n^{\rm ref} + \overline b(X_n^{\rm ref}) \Delta t.
\end{equation}

The scheme~\eqref{eq:scheme-ref-av} plays the role of a reference scheme to illustrate the consistency of the limiting scheme~\eqref{eq:scheme-limit-av} with the limiting equation, and to illustrate the fact that the crude scheme defined by~\eqref{eq:scheme-nonAP-av} fails to capture the correct limit and is not AP.

In Figure~\ref{fig:Av1_eps-to-0}, we represent the evolution of $X_{n}^\epsilon$, $X_n$ and $X_{n}^{\rm ref}$ as time $t_n=n\Delta t$ evolves, with $\Delta t = 0.004$ and for different values of $\epsilon$. In Figure~\ref{fig:Av1_AP-limit-ref}, $X_n^\epsilon$ and $X_n$ are computed using the AP scheme~\eqref{eq:scheme-AP-av} and the limit scheme~\eqref{eq:scheme-limit-av}, while in Figure~\ref{fig:Av1_nonAP-ref}, $X^\epsilon_n$ is computed using the crude scheme~\eqref{eq:scheme-nonAP-av}. Observe that, in both case, the scheme converges when $\epsilon \to 0$ and that the AP scheme~\eqref{eq:scheme-AP-av} does capture the correct limiting equation only with AP scheme~\eqref{eq:scheme-AP-av}, as opposed to the crude scheme~\eqref{eq:scheme-nonAP-av}.

\begin{figure}[htbp]
    \subfigure[AP scheme~\eqref{eq:scheme-AP-av} and its limit~\eqref{eq:scheme-limit-av}\label{fig:Av1_AP-limit-ref}]{
        \includegraphics[width=0.45\textwidth]{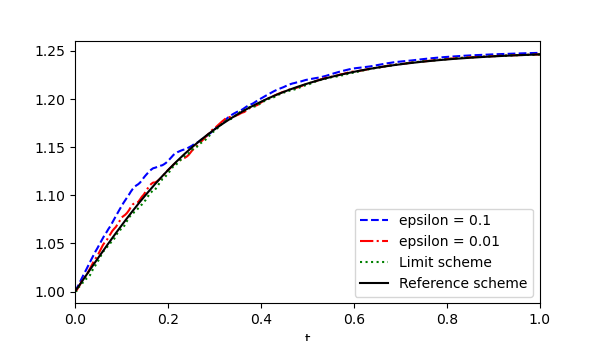}
    }
    \subfigure[Crude scheme~\eqref{eq:scheme-nonAP-av}\label{fig:Av1_nonAP-ref}]{
        \includegraphics[width=0.45\textwidth]{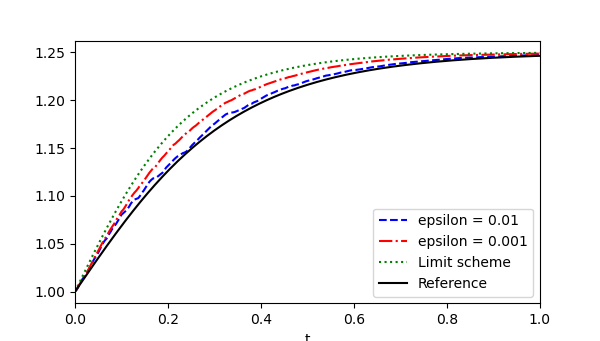}
    }
    \caption{Evolution of the AP scheme~\eqref{eq:scheme-AP-av} (left), the crude scheme~\eqref{eq:scheme-nonAP-av} (right), and the reference scheme~\eqref{eq:scheme-ref-av} (averaging regime), with $\Delta t = 0.004$.\label{fig:Av1_eps-to-0}}
\end{figure}

In Figure~\ref{fig:Av1_AP-nonAP-ref}, we represent the evolution of $X_{n}^\epsilon$ and $X_{n}^{\rm ref}$ as time $t_n=n\Delta t$ evolves, with $\Delta t = 0.004$ and $\epsilon = 0.001$, when $X_n$ is computed using the AP scheme~\eqref{eq:scheme-AP-av} or the crude scheme~\eqref{eq:scheme-nonAP-av}. It illustrates the superiority of the AP scheme over the crude scheme for a small $\epsilon$.

\begin{figure}[htbp]
    \centering
    \includegraphics[width=0.5\textwidth]{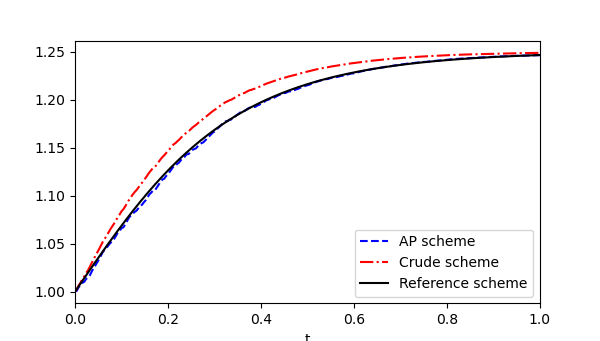}
    \caption{Evolution of the AP scheme~\eqref{eq:scheme-AP-av}, the crude scheme~\eqref{eq:scheme-nonAP-av} and the reference scheme~\eqref{eq:scheme-ref-av} (averaging regime), with $\Delta t = 0.004$ and $\epsilon = 0.001$.\label{fig:Av1_AP-nonAP-ref}}
\end{figure}

\subsection{Illustration in the diffusion approximation regime}

As in the previous section, the objective of this section is to illustrate qualitatively the superiority of the AP scheme~\eqref{eq:scheme-AP-diff-general} proposed in Section~\ref{sec:num-diff}, when the parameter $\epsilon$ is small, compared a not AP scheme.

The two examples described in Sections~\ref{sec:models-diff-ex} are considered below.

\subsubsection{First example}

Let us consider the first example, see Equation~\eqref{eq:SDE-diff-ex1}. The diffusion coefficient is given by $\sigma(x)=\cos(2\pi x)$. Let $T=1$, $x_0^\epsilon=1$ and $m_0^\epsilon=0$.

Recall that the AP scheme derived from the general case~\eqref{eq:scheme-AP-diff-general} in this case is given by~\eqref{eq:scheme-AP-diff-ex1}, the limiting scheme is given by~\eqref{eq:scheme-limit-diff-ex1} and the limiting equation is given by~\eqref{eq:SDE-limit-diff-ex1}. Let us define $X_n^{\rm ref}$ using the standard Euler-Maruyama scheme applied to this limiting equation (rewritten in It\^o form):
\begin{equation}\label{eq:scheme-ref-diff-ex1}
    X_{n+1}^{\rm ref} = X_n^{\rm ref} + \frac 1 2 \sigma(X_n^{\rm ref}) \sigma'(X_n^{\rm ref}) \Delta t + \sigma(X_n^{\rm ref}) \sqrt{\Delta t} \gamma_n.
\end{equation}

The scheme~\eqref{eq:scheme-ref-diff-ex1} plays the role of a reference scheme to illustrate the consistency of the limiting scheme~\eqref{eq:scheme-limit-diff-ex1} with the limiting equation, and to illustrate the fact that the crude scheme defined by~\eqref{eq:scheme-nonAP-diff} fails to capture the correct limit and is not AP.

In Figure~\ref{fig:Diff1_eps-to-0}, we represent the evolution of $X_{n}^\epsilon$ and $X_{n}^{\rm ref}$ as time $t_n=n\Delta t$ evolves, with $\Delta t = 0.004$ and for different values of $\epsilon$. The discretization $X^\epsilon_n$ is computed using the AP scheme~\eqref{eq:scheme-AP-diff-ex1} in Figure~\ref{fig:Diff1_AP-ref} and the crude scheme~\eqref{eq:scheme-nonAP-diff} in Figure~\ref{fig:Diff1_nonAP-ref}. Observe that, in both case, the scheme seem to converge when $\epsilon \to 0$ but only the AP scheme~\eqref{eq:scheme-AP-diff-ex1} captures the correct limiting.

\begin{figure}[htbp]
    \subfigure[AP scheme~\eqref{eq:scheme-AP-diff-ex1}\label{fig:Diff1_AP-ref}]{
        \includegraphics[width=0.45\textwidth]{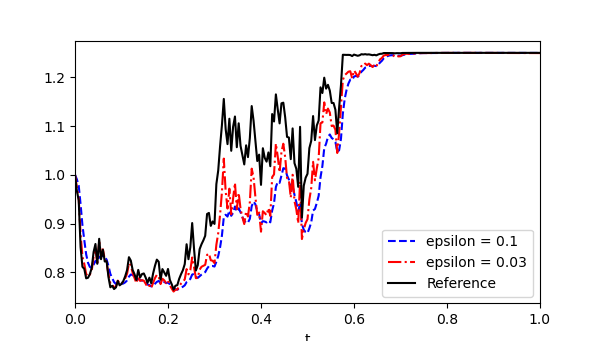}
    }
    \subfigure[Crude scheme~\eqref{eq:scheme-nonAP-diff}\label{fig:Diff1_nonAP-ref}]{
        \includegraphics[width=0.45\textwidth]{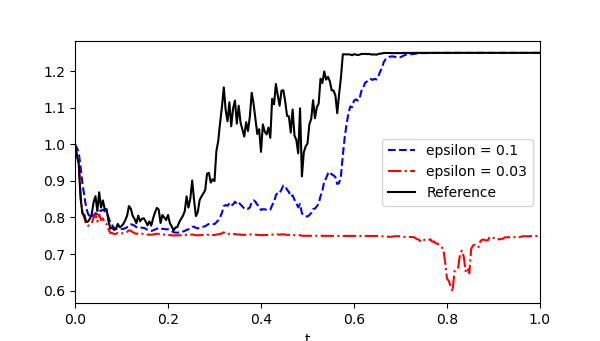}
    }
    \caption{Evolution of the AP scheme~\eqref{eq:scheme-AP-diff-ex1} (left), the crude scheme~\eqref{eq:scheme-nonAP-diff} (right), and the reference scheme~\eqref{eq:scheme-ref-diff-ex1} (diffusion approximation regime regime, first example), with $\Delta t = 0.004$.\label{fig:Diff1_eps-to-0}}
\end{figure}

In Figure~\ref{fig:Diff1_AP-nonAP-ref}, we represent the evolution of $X_{n}^\epsilon$ and $X_{n}^{\rm ref}$ as time $t_n=n\Delta t$ evolves, with $\Delta t = 0.004$ and $\epsilon = 0.01$, when $X_n^\epsilon$ is computed using the AP scheme~\eqref{eq:scheme-AP-diff-ex1} or the crude scheme~\eqref{eq:scheme-nonAP-diff}. Note how the behavior of the crude scheme differs from the reference. It reveals the superiority of the AP scheme for a small $\epsilon$.

\begin{figure}[htbp]
    \centering
    \includegraphics[width=0.5\textwidth]{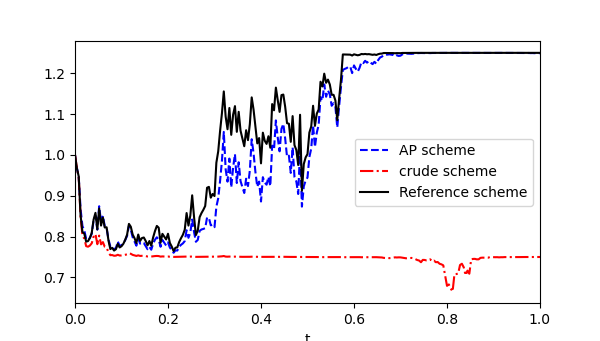}
    \caption{Evolution of the AP scheme~\eqref{eq:scheme-AP-diff-ex1}, the crude scheme~\eqref{eq:scheme-nonAP-diff} and the reference scheme~\eqref{eq:scheme-ref-diff-ex1} (diffusion approximation regime, first example), with $\Delta t = 0.004$ and $\epsilon = 0.01$.\label{fig:Diff1_AP-nonAP-ref}}
\end{figure}

\subsubsection{First example with $\sigma(x) = x$}

In this section, we illustrate the performance of the AP scheme presented in Remark~\ref{rem:ex1bis}, and an important feature of all the AP schemes presented in this article, concerning the consistency of quadrature rules for discretizations of the fast component.

As explained in Remark~\ref{rem:ex1bis}, when $\sigma(x)=x$, where $x \in \R$ belongs to the real line instead of imposing periodic conditions, another type of AP scheme~\eqref{eq:scheme-AP-ex1bis} can be designed. The limiting equation is $dX_t = X_t \circ dW_t$ or, with an It\^o convention, $dX_t = \frac 1 2 X_t dt + X_t dW_t$, and the Euler-Maruyama scheme (used as a reference scheme) for this limiting equation is written as
\begin{equation}\label{eq:scheme-ref-diff-ex1bis}
    X_{n+1}^{\rm ref} = X_n^{\rm ref} + \frac 1 2 X_n^{\rm ref} \Delta t + X_n^{\rm ref} \sqrt{\Delta t} \gamma_n.
\end{equation}

Recall that in~\eqref{eq:scheme-AP-ex1bis}, the quadrature rule used to discretize the integral in the exponential is closely related to the choice of the scheme for the discretization of the fast component. Let us introduce the following scheme where  the consistency is not satisfied (scheme~\eqref{eq:scheme-AP-ex1bis} corresponds to $\theta = \theta'$ below):
\begin{equation} \label{eq:scheme-nonAP-diff-ex1bis}
    \left\lbrace
    \begin{aligned}
        X_{n+1}^\epsilon&=X_n^\epsilon \exp\bigl(\frac{\Delta t}{\epsilon}[(1-\theta)m_n^\epsilon+\theta m_{n+1}^\epsilon]\bigr)\\
        m_{n+1}^\epsilon&=m_n^\epsilon-\frac{\Delta t}{\epsilon^2}\bigl[(1-\theta')m_n^\epsilon+\theta' m_{n+1}^\epsilon\bigr]+\frac{\sqrt{\Delta t}}{\epsilon}\gamma_n,
    \end{aligned}
    \right.
\end{equation}
with $\theta, \theta' \in [\frac 1 2, 1]$.

In Figure~\ref{fig:Diff1x_AP-APgeneral-nonAP-ref}, we represent the evolution of $X_{n}^\epsilon$ and $X_{n}^{\rm ref}$ as time evolves, with $\Delta t = 0.004$ and $\epsilon = 0.01$. In Figure~\ref{fig:Diff1x_AP-APgeneral-ref} $X_n^\epsilon$ is computed either the specific AP scheme~\eqref{eq:scheme-AP-ex1bis} or the general AP scheme~\eqref{eq:scheme-AP-diff-general}, while in Figure~\ref{fig:Diff1x_nonAP-ref}, it is computed using the scheme~\eqref{eq:scheme-nonAP-diff-ex1bis} above with $\theta = 1 \neq \theta' = 0.5$. It illustrates the AP property of both schemes~\eqref{eq:scheme-AP-ex1bis} and~\eqref{eq:scheme-AP-diff-ex1} and the non convergence when the quadrature rules are not chosen consistently.

\begin{figure}[htbp]
    \subfigure[AP schemes~\eqref{eq:scheme-AP-ex1bis} and~\eqref{eq:scheme-AP-diff-ex1}\label{fig:Diff1x_AP-APgeneral-ref}]{
        \includegraphics[width=0.45\textwidth]{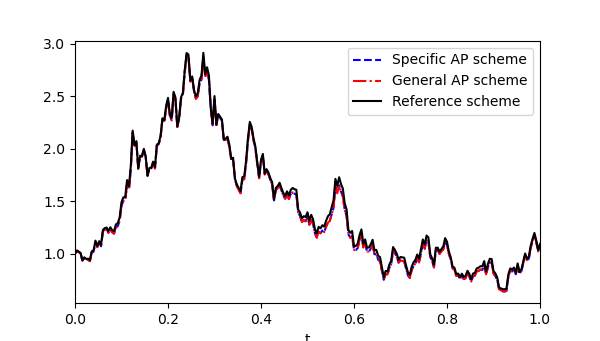}
    }
    \subfigure[Scheme~\eqref{eq:scheme-nonAP-diff-ex1bis} with $\theta = 1 \neq \theta' = 0.5$\label{fig:Diff1x_nonAP-ref}]{
        \includegraphics[width=0.45\textwidth]{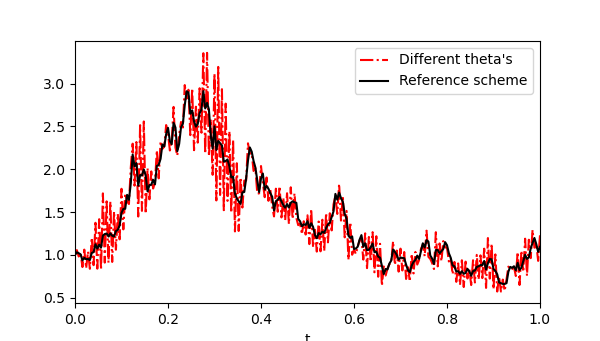}
    }
    \caption{Evolution of the AP schemes~~\eqref{eq:scheme-AP-ex1bis} and~\eqref{eq:scheme-AP-diff-ex1} (left), the crude scheme~\eqref{eq:scheme-nonAP-diff-ex1bis} with $\theta\neq \theta'$ (right), and the reference scheme~\eqref{eq:scheme-ref-diff-ex1bis} (diffusion approximation regime, first example with $\sigma(x) = x$), $\Delta t = 0.004$ and $\epsilon = 0.01$.\label{fig:Diff1x_AP-APgeneral-nonAP-ref}}
\end{figure}

\subsubsection{Second example}

Let us now consider the second example described in Section~\ref{sec:models-diff-ex}, see Equation~\eqref{eq:SDE-diff-ex2}. The coefficients are given by $f(x)=\cos(2\pi x)+1.5$, $g(x) = 0$ and $h(x) = 1$. Let $T=1$, $x_0^\epsilon=1$ and $m_0^\epsilon=0$.

The general case~\eqref{eq:scheme-AP-diff-general} gives in this case the AP scheme~\eqref{eq:scheme-AP-diff-ex2} and the limiting scheme~\eqref{eq:scheme-limit-diff-ex2}, whereas the limiting equation is given by~\eqref{eq:SDE-limit-diff-ex2}. The reference scheme is obtained by using the standard Euler-Maruyama scheme applied to the limiting equation:
\begin{equation}\label{eq:scheme-ref-diff-ex2}
    X_{n+1}^{\rm ref} = X_n^{\rm ref}  - \Delta t\frac{h(X_n^{\rm ref})^2 f'(X_n^{\rm ref})}{2 f(X_n^{\rm ref})}   +  \sqrt{\Delta t} \gamma_n.
\end{equation}

We represent in Figure~\ref{fig:Diff2_AP-nonAP-ref} the evolution of $X_{n}^\epsilon$ and $X_{n}^{\rm ref}$ as time evolves, with $\Delta t = 0.004$ and $\epsilon = 0.01$, where $X_n^\epsilon$ is computed using the AP scheme~\eqref{eq:scheme-AP-diff-ex2} (left) and the crude scheme~\eqref{eq:scheme-nonAP-diff} (right). Observe that the AP scheme captures the correct limiting equation when $\epsilon \to 0$, whereas the crude scheme does not.

\begin{figure}[htbp]
    \centering
    \includegraphics[width=0.5\textwidth]{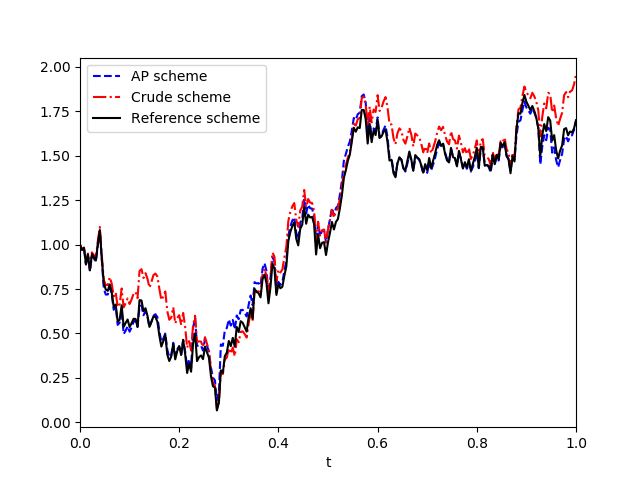}
    \caption{Evolution of the AP scheme~\eqref{eq:scheme-AP-diff-ex2}, the crude scheme~\eqref{eq:scheme-nonAP-diff}, and the reference scheme~\eqref{eq:scheme-ref-diff-ex2} for the second example~\eqref{eq:SDE-diff-ex2} (diffusion approximation regime, second example), with $\Delta t = 0.004$ and $\epsilon = 0.01$.\label{fig:Diff2_AP-nonAP-ref}}
\end{figure}

\section{Proof of Theorem~\ref{th:UA-av}}\label{sec:error}

The objective of this section is to prove the error estimate~\eqref{eq:UA-1}. The proof follows from proving the following four auxiliary lemmas. In their statements, let Assumptions~\ref{ass:init},~\ref{ass:coeffs-av} and~\ref{ass:av} be satisfied. Let $T\in(0,\infty)$ be fixed and assume that $\varphi:\T^d\to\R$ is of class $\mathcal{C}^4$. Recall that the identity $T=N\Delta t$ is assumed to hold. In addition, recall that $\bigl(X(t)\bigr)_{t\ge 0}$ and $\bigl(X_n\bigr)_{n\ge 0}$ are defined by the limiting equation~\eqref{eq:limit_SDE-av} and the limiting scheme~\eqref{eq:scheme-limit-av} respectively.
\begin{lemma}\label{lem1}
There exists $C(T,\varphi)\in(0,\infty)$ such that for all $\Delta t\in(0,\Delta t_0]$ and $\epsilon\in(0,1]$ one has
\begin{equation}\label{eq:lem1}
\big|\E[\varphi(X_N^\epsilon)]-\E[\varphi(X^\epsilon(T))]\big|\le C(T,\varphi)\frac{\Delta t}{\epsilon}.
\end{equation}
\end{lemma}

\begin{lemma}\label{lem2}
There exists $C(T,\varphi)\in(0,\infty)$ such that for all $\epsilon\in(0,1]$ one has
\begin{equation}\label{eq:lem2}
\big|\E[\varphi(X(T))]-\E[\varphi(X^\epsilon(T))]\big|\le C(T,\varphi)\epsilon.
\end{equation}
\end{lemma}

\begin{lemma}\label{lem3}
There exists $C(T,\varphi)\in(0,\infty)$ such that for all $\Delta t\in(0,\Delta t_0]$ one has
\begin{equation}\label{eq:lem3}
\big|\E[\varphi(X_N)]-\E[\varphi(X(T))]\big|\le C(T,\varphi)\Delta t.
\end{equation}
\end{lemma}

\begin{lemma}\label{lem4}
There exists $C(T,\varphi)\in(0,\infty)$ such that for all $\Delta t\in(0,\Delta t_0]$ and $\epsilon\in(0,1]$ one has
\begin{equation}\label{eq:lem4}
\big|\E[\varphi(X_N^\epsilon)]-\E[\varphi(X_N)]\big|\le C(T,\varphi)\max(\Delta t,\epsilon).
\end{equation}
\end{lemma}

The first auxiliary result (Lemma~\ref{lem1}) states a weak error estimate for the numerical scheme~\eqref{eq:scheme-AP-av} for fixed $\epsilon\in(0,1]$. Due to the stiffness of the fast component $m_n^{\epsilon}$, the right-hand side is not uniform with respect to $\epsilon$, and it is natural to expect that the upper bound depends on $\tau=\frac{\Delta t}{\epsilon}$.

The second auxiliary result (Lemma~\ref{lem2}) gives an error estimate in the averaging principle (see~\eqref{eq:error-av} in Proposition~\ref{propo:limit-SDE-av}), in the weak sense. This is a standard result in the literature, see for instance~\cite{KhasminskiiYin:05} for an approach using asymptotic expansions for solutions of Kolmogorov equations. The strategy of the proof provided in Section~\ref{sec:app-av} is based on the introduction of the solutions of relevant Poisson equations, in the spirit of~\cite[Chapter~17]{PavliotisStuart} where strong convergence is studied, see~\cite{Brehier:20} and~\cite{RocknerSunXie:19} for the weak convergence case.

The two remaining auxiliary lemmas and their proofs are more original than the first two. Lemmas~\ref{lem3} and~\ref{lem4} are quantitative statements concerning two fundamental requirements in the notion of AP scheme (see Definition~\ref{def:AP}). On the one hand, Lemma~\ref{lem3} is a quantitative statement of the consistency of the limiting scheme~\eqref{eq:scheme-limit-av} with the limiting equation~\eqref{eq:limit_SDE-av}, since it provides a weak error when $\Delta t\to 0$. Since the scheme is not classical (it is not a standard Euler-Maruyama type method, in particular recall that the scheme is random even if $X(T)$ is deterministic, when $\sigma=0$), a proof is required. On the other hand, Lemma~\ref{lem4} is a quantitative statement about the convergence to the limiting scheme, for fixed $\Delta t\in(0,\Delta t_0]$ (see Assumption~\ref{ass:limiting_scheme}). In fact, the left-hand side of~\eqref{eq:lem4} goes to $0$ when $\epsilon\to 0$, however in the right-hand side of~\eqref{eq:lem4} an additional error term $\Delta t$ appears. Proving Lemma~\ref{lem4} is the most challenging step towards the proof of Theorem~\ref{th:UA-av}, whereas a key argument will be identified in the proof of Lemma~\ref{lem3} related to the consistency of the limiting scheme with the limiting equation.

The following auxiliary results concerning solutions of Kolmogorov equations are required in order do prove the four auxiliary results stated above.
\begin{lemma}\label{lem:Kolmo1}
Define $u^\epsilon(t,x)=\E_{x,m}[\varphi(X^\epsilon(t))]$, for all $t\ge 0$, $x\in\T^d$ and $m\in\R$, where $\bigl(X^\epsilon(t),m^\epsilon(t)\bigr)_{t\ge 0}$ is the solution of the SDE system~\eqref{eq:SDE-av}, and $\E_{x,m}$ means that $X^\epsilon(0)=x$ and $m^\epsilon(0)=m$. For all $\epsilon\in(0,1]$, one has $u^\epsilon\in\mathcal{C}([0,T],\mathcal{C}_b^3(\T^d\times\R,\R))$. In addition, there exists $C(T,\varphi)\in(0,\infty)$ such that for all $j\in\{1,2,3\}$, one has
\begin{equation}\label{eq:Kolmo1}
\underset{\epsilon\in(0,1]}\sup~\underset{(t,x,m)\in [0,T]\times\T^d\times\R}\sup~\|D_x^ju^\epsilon(t,x,m)\|\le C(T,\varphi).
\end{equation}
\end{lemma}

\begin{lemma}\label{lem:Kolmo2}
Define $u(t,x)=\E_{x}[\varphi(X(t))]$, for all $t\ge 0$ and $x\in\T^d$, where $\bigl(X(t)\bigr)_{t\ge 0}$ is the solution of the SDE~\eqref{eq:limit_SDE-av} and $\E_x$ means that $X(0)=x$. One has $u\in\mathcal{C}([0,T],\mathcal{C}_b^4(\T^d,\R))$. In addition, there exists $C(T,\varphi)\in(0,\infty)$ such that for all $j\in\{1,2,3,4\}$, one has
\begin{equation}\label{eq:Kolmo2}
\underset{(t,x)\in [0,T]\times\T^d}\sup~\|D_x^ju(t,x)\|\le C(T,\varphi).
\end{equation}
\end{lemma}

\begin{lemma}\label{lem:Kolmo3}
Let $\Delta t\in(0,\Delta t_0]$.

Define $u_n(x)=\E_x[\varphi(X_n)]$, for all $n\in\N$ and $x\in \T^d$, where $\bigl(X_n\bigr)_{n\ge 0}$ is defined by the limiting scheme~\eqref{eq:scheme-limit-av} (see Theorem~\ref{th:AP-av}), and $\E_x$ means that $X_0=x$. For all $n\ge 0$, one has $u_n\in\mathcal{C}^2(\T^d)$. In addition, there exists $C(T,\varphi)\in(0,\infty)$ such that for all $j\in\{1,2\}$, one has
\begin{equation}\label{eq:Kolmo3-1}
\underset{0\le n\le N}\sup~\underset{x\in \T^d}\sup~\|D_x^ju_n(x)\|\le C(T,\varphi)
\end{equation}
and, for all $\Delta t\in(0,\Delta t_0]$, one has
\begin{equation}\label{eq:Kolmo3-2}
\underset{0\le n\le N-1}\sup~\underset{x\in\T^d}\sup~\|D_x^ju_{n+1}(x)-D_x^ju_n(x)\|\le C(T,\varphi)\Delta t.
\end{equation}
\end{lemma}

Based on the auxiliary results stated above, the proof of Theorem~\ref{th:UA-av} is straightforward.
\begin{proof}[Proof of Theorem~\ref{th:UA-av}]
Note that
\begin{align*}
\big|\E[\varphi(X_N^\epsilon)-\E[\varphi(X^\epsilon(T))]\big|&\le \big|\E[\varphi(X(T))-\E[\varphi(X^\epsilon(T))]\big|\\
&\quad+\big|\E[\varphi(X_N)-\E[\varphi(X(T))]\big|\\
&\quad+\big|\E[\varphi(X_N^\epsilon)-\E[\varphi(X_N)]\big|,
\end{align*}
thus combining~\eqref{eq:lem2},~\eqref{eq:lem3} and~\eqref{eq:lem4} (with $\max(\Delta t,\epsilon)\le \Delta t+\epsilon$) yields
\[
\big|\E[\varphi(X_N^\epsilon)-\E[\varphi(X^\epsilon(T))]\big|\le C(T,\varphi)\bigl(\Delta t+\epsilon\bigr).
\]
Combining that error estimate with~\eqref{eq:lem1} then concludes the proof of the error estimate~\eqref{eq:UA-1}. As already explained above, the error estimate~\eqref{eq:UA-2} is a straightforward consequence of~\eqref{eq:UA-1} (considering the cases $\sqrt{\Delta t}\le \epsilon$ and $\sqrt{\Delta t}\ge \epsilon$).

This concludes the proof of Theorem~\ref{th:UA-av}.
\end{proof}

Let us now give proofs of the auxiliary lemmas~\ref{lem1},~\ref{lem2},~\ref{lem3} and~\ref{lem4}, employing the results of Lemmas~\ref{lem:Kolmo1},~\ref{lem:Kolmo2} and~\ref{lem:Kolmo3} (proofs are given below).

The following notation is used below in the proofs of the auxiliary results: for all $\lambda,\mu\ge 0$, $\lambda \lesssim \mu$ means that there exists $C(T,\varphi) \in (0,+\infty)$, independent of $\Delta t$, $\epsilon$ and $n$, such that $\lambda \leq C(T,\varphi)\mu$. In addition, the following notation is used for the infinitesimal generator of the Ornstein-Uhlenbeck process:
\begin{equation*}
    \L_{OU}^x = - m \partial_m + h(x)^2 \partial_m^2,
\end{equation*}
in order to let the dependence with respect to $x$ be clear.

\begin{proof}[Proof of Lemma~\ref{lem1}]
Let us introduce auxiliary continuous-time processes $\tilde{X}^\epsilon$ and $\tilde{m}^\epsilon$
, such that, for all $n\in\{0,\ldots,N\}$, one has $X^\epsilon_n = \tilde X^\epsilon(t_n)$ and $m^\epsilon_n = \tilde m^\epsilon(t_n)$ (recall the definition~\eqref{eq:scheme-AP-av} of the scheme): for $t_n \leq t \leq t_{n+1}$
\begin{align*}
    \tilde X^\epsilon(t) &= X^\epsilon_n + (t-t_n) b(X^\epsilon_n, \tilde m^\epsilon(t)) + \sigma(X^\epsilon_n,\tilde m^\epsilon(t)) \left(B(t) - B(t_n)\right)\\
    d\tilde m^\epsilon_t &= - \frac{\tilde m^\epsilon_t}{\epsilon} dt + \frac{\sqrt{2}h(X^\epsilon_n)}{\sqrt{\epsilon}} d\beta_t.
\end{align*}
Note that $\tilde m^\epsilon$ does satisfy $m^\epsilon_n = \tilde m^\epsilon(t_n)$, since $m^\epsilon_n$ is exact in distribution and $\tilde m^\epsilon$ is an Ornstein-Uhlenbeck process with variance $\frac{2 h(X^\epsilon_n)^2}{\epsilon}$. The process $\tilde{X}^\epsilon$ satisfies on each subinterval $[t_n,t_{n+1}]$ the following stochastic differential equation: for all $t\in[t_n,t_{n+1}]$, one has
    \begin{align*}
        d\tilde X^\epsilon_t &= b(X^\epsilon_n, \tilde m^\epsilon_t) dt + \sigma(X^\epsilon_n, \tilde m^\epsilon_t) dB_t + (t-t_n) \partial_m b(X^\epsilon_n, \tilde m^\epsilon_t) d\tilde m^\epsilon_t\\
        &\phantom{=} + \partial_m \sigma(X^\epsilon_n, \tilde m^\epsilon_t) (B_t - B_{t_n}) d\tilde m^\epsilon_t + \frac{h(X^\epsilon_n)^2}{\epsilon} (t-t_n) \partial_m^2 b(X^\epsilon_n,\tilde m^\epsilon_t) dt\\
        &\phantom{=} + \frac{h(X^\epsilon_n)^2}{\epsilon} \partial_m^2 \sigma(X^\epsilon_n, \tilde m^\epsilon_t) (B_t - B_{t_n}) dt.
    \end{align*}
The expressions for $\tilde{X}^\epsilon$ are complicated due to the fact that in the scheme~\eqref{eq:scheme-AP-av}, $b$ and $\sigma$ are evaluated with $m=m^\epsilon_{n+1}$, which is required to satisfy the AP property.

    Owing to Lemma~\ref{lem:Kolmo1}, the auxiliary function $u^\epsilon$ is of class $\mathcal{C}^2$ and is solution of the Kolmogorov equation $\partial_tu^\epsilon=\L^\epsilon u^\epsilon$. Using a telescoping sum argument and the definition of the auxiliary processes $\tilde{X}^\epsilon$ and $\tilde{m}^\epsilon$, the application of It\^o's formula yields the following standard expression for the weak error:
    \begin{align*}
        \E[\varphi(X_N^\epsilon)] - &\E[\varphi(X^\epsilon(T))] = \E[u^\epsilon(0,X_N^\epsilon,\tilde m^\epsilon(t_N))] - \E[u^\epsilon(T,X_0^\epsilon,\tilde m^\epsilon(0))]\\
        &= \sum_{n=0}^{N-1} \E[u^\epsilon(T-t_{n+1},\tilde{X}^\epsilon(t_{n+1}),\tilde m^\epsilon(t_{n+1})) - u^\epsilon(T-t_n,\tilde{X}^\epsilon(t_n),\tilde m^\epsilon(t_n))]\\
        &= \sum_{n=0}^{N-1} \int_{t_n}^{t_{n+1}} \E[(-\partial_t + \tilde \L_n^\epsilon) u^\epsilon(T-t,\tilde X^\epsilon(t),\tilde m^\epsilon(t))]dt\\
        &= \sum_{n=0}^{N-1} \int_{t_n}^{t_{n+1}} \E[(\tilde \L_n^\epsilon - \L^\epsilon) u^\epsilon(T-t,\tilde X^\epsilon(t),\tilde m^\epsilon(t))]dt,
    \end{align*}
    where the auxiliary differential operator $\tilde{\L}^\epsilon_n$ is such that
    \begin{align*}
        \tilde \L^\epsilon_n u^\epsilon(T-t,\tilde X^\epsilon(t),\tilde m^\epsilon(t)) &= b(X^\epsilon_n,\tilde m^\epsilon(t)) \cdot \nabla_x u^\epsilon(T-t,\tilde X^\epsilon(t),\tilde m^\epsilon(t))\\
        &\phantom{=} + \frac{1}{2} \sigma \sigma^* (X^\epsilon_n,\tilde m^\epsilon(t)) : \nabla_x^2 u^\epsilon(T-t,\tilde X^\epsilon(t),\tilde m^\epsilon(t))\\
        &\phantom{=} + \frac{1}{\epsilon} \L_{OU}^{X^\epsilon_n} u^\epsilon(T-t,\tilde X^\epsilon(t),\tilde m^\epsilon(t)) + r^\epsilon_n(t),
    \end{align*}
    and where the remainder term $r^\epsilon_n(t)$ is given by
    \begin{align*}
        r^\epsilon_n(t) &= \frac{1}{\epsilon} (t-t_n) \L_{OU}^{X^\epsilon_n} b({X^\epsilon_n},\tilde m^\epsilon(t)) \cdot \nabla_x u^\epsilon\\
        &\phantom{=} + \frac{1}{\epsilon} \L_{OU}^{X^\epsilon_n} \sigma({X^\epsilon_n},\tilde m^\epsilon(t)) (B(t) - B(t_n)) \cdot \nabla_x u^\epsilon\\
        &\phantom{=} + \frac{1}{2} D_x^2 u^\epsilon \cdot \left( (t-t_n) \partial_m b({X^\epsilon_n},\tilde m^\epsilon(t)) + \partial_m \sigma({X^\epsilon_n},\tilde m^\epsilon(t))(B(t)-B(t_n)) \right)^2,
    \end{align*}
    where $u^\epsilon$ and its derivatives are evaluated at $(T-t,\tilde X^\epsilon(t),\tilde m^\epsilon(t))$, and where we used the notation $D^2 \phi (x) \cdot (y)^2 = y y^* : \nabla_x^2 \phi(x)$ to simplify the presentation.

Let us first deal with the remainder term $r_n^\epsilon(t)$. Observe that the processes $\bigl(B(t)-B(t_n)\bigr)_{t\in[t_n,t_{n+1}]}$ and $\bigl\{X_n^{\epsilon},\bigl(\tilde{m}^\epsilon(t)\bigr)_{t\in[t_n,t_{n+1}]}\bigr\}$ are independent, thus using a conditioning argument and the regularity estimates from Lemma~\ref{lem:Kolmo1}, one has for all $t\in[t_n,t_{n+1}]$
\[
\abs{\E[r^\epsilon_n(t)]} \lesssim \frac{t-t_n}{\epsilon}.
\]

It remains to deal with
\begin{equation}\label{eq:tilde_L_n_epsilon-L_epsilon}
    \begin{aligned}
        (\tilde \L_n^\epsilon - &\L^\epsilon) u^\epsilon(T-t,\tilde X^\epsilon(t),\tilde m^\epsilon(t))-r^\epsilon_n(t)\\
         &= \left( b(X^\epsilon_n,\tilde m^\epsilon(t)) - b(\tilde X^\epsilon(t),\tilde m^\epsilon(t)) \right) \cdot \nabla_x u^\epsilon(T-t,\tilde X^\epsilon(t),\tilde m^\epsilon(t))\\
        &+ \frac 1 2 \left( \sigma \sigma^* (X^\epsilon_n,\tilde m^\epsilon(t)) - \sigma \sigma^* (\tilde X^\epsilon(t),\tilde m^\epsilon(t)) \right) : \nabla_x^2 u^\epsilon(T-t,\tilde X^\epsilon(t),\tilde m^\epsilon(t))\\
        &+ \left( h(X^\epsilon_n)^2 - h(\tilde X^\epsilon(t))^2 \right) \partial_m^2 u^\epsilon(T-t,\tilde X^\epsilon(t),\tilde m^\epsilon(t)),\\
    \end{aligned}
\end{equation}
where the expressions~\eqref{eq:L^epsilon-av} and~\eqref{eq:L_01-av} for the infinitesimal generator $\L^\epsilon$ have been used. The three quantities appearing in the right-hand side of~\eqref{eq:tilde_L_n_epsilon-L_epsilon} above are of the type
    \begin{equation*}
        \left( V(\tilde X^\epsilon(t), \tilde m^\epsilon(t)) - V(X^\epsilon_n,\tilde m^\epsilon(t)) \right) U^\epsilon(\tilde X^\epsilon(t),\tilde m^\epsilon(t)),
    \end{equation*}
    for $V = b$, $\sigma \sigma^*$ or $h^2$ and $U^\epsilon = \nabla_x u^\epsilon(T-t)$, $\nabla_x^2 u^\epsilon(T-t)$ or $\partial_m^2 u^\epsilon(T-t)$. Using again the independence of $(B(t)-B(t_n))_{t_n \leq t \leq t_{n+1}}$ and $\bigl\{X_n^{\epsilon},\bigl(\tilde{m}^\epsilon(t)\bigr)_{t\in[t_n,t_{n+1}]}\bigr\}$ and regularity properties of $u^\epsilon$ given in Lemma~\ref{lem:Kolmo1}, applying It\^o's formula and conditioning with respect to $(X^\epsilon_n, (\tilde m^\epsilon(t))_{t_n \leq t \leq t_{n+1}})$, one obtains
    \begin{equation*}
        \abs{\E\left[ \left( V(\tilde X^\epsilon(t), \tilde m^\epsilon(t)) - V(X^\epsilon_n,\tilde m^\epsilon(t)) \right) U^\epsilon(X^\epsilon_n,\tilde m^\epsilon(t)) \right]} \lesssim t-t_n.
    \end{equation*}
    Moreover, $\norm{V}_{\C^1(\T^d, C(\R))} \lesssim 1$ and, owing to Lemma~\ref{lem:Kolmo1}, $\norm{U^\epsilon}_{\C^1(\T^d, C(\R))} \lesssim 1$. Therefore, we have
    \begin{align*}
        \E &\left[ \abs{ \left( V(\tilde X^\epsilon(t), \tilde m^\epsilon(t)) - V(X^\epsilon_n,\tilde m^\epsilon(t)) \right) \left( U^\epsilon(\tilde X^\epsilon(t), \tilde m^\epsilon(t)) - U^\epsilon(X^\epsilon_n,\tilde m^\epsilon(t)) \right) } \right]\\
        &\leq \norm{V(\cdot,\tilde m^\epsilon(t))}_{\C^1(\T^d)} \E \left[ \norm{U^\epsilon(\cdot,\tilde m^\epsilon(t))}_{\C^1(\T^d)} \norm{\tilde X^\epsilon(t) - X^\epsilon_n}^2 \right] \\
        &\lesssim \E \left[ \norm{\tilde X^\epsilon(t) - X^\epsilon_n}^2 \right] \lesssim t-t_n,
    \end{align*}
where the last inequality comes from the definition of $\tilde{X}^\epsilon(t)$. Gathering the two estimates above gives, for all $t\in[t_n,t_{n+1}]$,
    \begin{equation*}
        \abs{\E \left[ \left( V(\tilde X^\epsilon(t), \tilde m^\epsilon(t)) - V(X^\epsilon_n,\tilde m^\epsilon(t)) \right) U^\epsilon(\tilde X^\epsilon(t),\tilde m^\epsilon(t)) \right]} \lesssim t-t_n,
    \end{equation*}
and using~\eqref{eq:tilde_L_n_epsilon-L_epsilon} finally yields, for all $t\in[t_n,t_{n+1}]$,
    \begin{align*}
        \E\left[\abs{(\tilde \L_n^\epsilon - \L^\epsilon) u^\epsilon(T-t,\tilde X^\epsilon(t),m^\epsilon(t))}\right] &\lesssim t-t_n + \frac{t-t_n}{\epsilon}\lesssim \frac{\Delta t}{\epsilon}.
    \end{align*}
One then obtains
\[
\abs{\sum_{n=0}^{N-1} \int_{t_n}^{t_{n+1}} \E[(\tilde \L_n^\epsilon - \L^\epsilon) u^\epsilon(T-t,\tilde X^\epsilon(t),\tilde m^\epsilon(t))]dt}\lesssim \frac{\Delta t}{\epsilon},
\]
which concludes the proof of~\eqref{eq:lem1} and of Lemma~\ref{lem1}.
\end{proof}

 \begin{proof}[Proof of Lemma~\ref{lem2}]
 For all $x\in\T^d$, introduce the auxiliary Ornstein-Uhlenbeck process $m^x$ solving the SDE
     \begin{equation*}
             d m^x_t = - m^x_t dt + \sqrt{2}h(x) d\beta_t
     \end{equation*}
 Let $m^x(t,m)$ denote the solution at time $t$, if the initial condition is given by $m^x(0,m)$. The invariant distribution of the process $m^x$ is equal to $\nu^x$. Note that for $m$ and $m' \in \R$, one has $m^x(t,m)-m^x(t,m') = (m-m') e^{-t}$. Consider $V=b$ or $V=\sigma\sigma^\star$, and let
\begin{equation*}
\delta(t,x,m)=\E[V(x,m^\epsilon(t,m))-\overline{V}(x)].
\end{equation*} 
Note that for $m$ and $m' \in \R$, one has $m^x(t,m)-m^x(t,m') = (m-m') e^{-t}$. As a consequence, we have
\begin{equation*}
    \norm{\delta(t,x,m) - \delta(t,x,m')} \lesssim \abs{m-m'} e^{-t}.
\end{equation*}
By integrating with respect to $m'$ and using the equality $\overline V(x) = \int V(x,m') d\nu^x(m')$, one obtains
\begin{equation}\label{eq:OU_exponential-mixing}
    \norm{\delta(t,x,m)} \lesssim (1+\abs{m}) e^{-t}.
\end{equation}
 Using the fact that $\delta$ and its derivatives satisfy~\eqref{eq:OU_exponential-mixing} with $V=b$, one is able to check that the function $\psi_b$ given by, for all $x\in\T^d$ and $m\in\R$,
     \begin{equation*}
         \psi_b(x,m) = - \int_0^\infty \E[b(x,m^x(t,m)) - \overline b(x)] dt
     \end{equation*}
 is well-defined (by definition of $\overline{b}$, see~\eqref{eq:averagedcoeffs1}) and is of class $\mathcal{C}^2(\T^d \times \R)$. Moreover, $\psi_b$ and its derivative have at most linear growth in $m$. In addition, $\psi_b(x,\cdot)$ solves the Poisson equation $\L_{OU}^x \psi_b(x,m) = b(x,m) - \overline b(x)$, for all $x\in\T^d$ (indeed $\nu^x$ is the invariant distribution associated with the generator $\L_{OU}^x$ of the Ornstein-Uhlenbeck process $m^x$).

 Similarly, define, for all $x\in\T^d$ and $m\in\R$,
     \begin{equation*}
         \psi_\sigma(x,m) = - \int_0^\infty \E[\sigma \sigma^*(x,m^x(t,m)) - \overline \sigma \, \overline \sigma^*(x)] dt.
     \end{equation*}
 The function $\psi_\sigma$ is well-defined: owing to~\eqref{eq:averagedcoeffs2} (Assumption~\ref{ass:coeffs-av}) one has the equality $\int \sigma \sigma^*(x,m)d\nu^x(m)=\overline \sigma \, \overline \sigma^*(x)$ for all $x\in\T^d$, and using the same arguments as above, $\psi_\sigma(x,\cdot)$ solves the Poisson equation $\L_{OU}^x\psi_\sigma(x,m)=\sigma \sigma^*(x,m) - \overline \sigma \, \overline \sigma^*(x)$, for all $x\in\T^d$.

 Now, for all $t\in[0,T]$, $x\in\T^d$ and $m\in\R$, let
 \begin{equation*}
     \Phi(t,x,m) = \psi_b(x,m) \cdot \nabla_x u(T-t,x) + \psi_\sigma(x,m) : \nabla_x^2 u(T-t,x),
 \end{equation*}
 where $u$ is given by Lemma~\ref{lem:Kolmo2} and solves the Kolmogorov equation $\partial_tu=\L u$.

 On the one hand, applying It\^o's formula yields the following expression for the error term:
     \begin{align*}
         \E[\varphi(X^\epsilon(T))] - &\E[\varphi(X(T))] = \E[u(0,X^\epsilon(T)) - u(T,X^\epsilon(0))]\\
         &= \int_{0}^{T} \E[\left(b(X^\epsilon(t),m^\epsilon(t)) - \overline{b}(X^\epsilon(t))\right) \cdot \nabla_x u(T-t,X^\epsilon(t))] dt\\
         &\phantom{=} + \frac{1}{2} \int_{0}^{T} \E[\left(\sigma \sigma^*(X^\epsilon(t),m^\epsilon(t)) - \overline \sigma \, \overline \sigma^*(X^\epsilon(t))\right) : \nabla_x^2 u(T-t,X^\epsilon(t))] dt\\
         &= \int_{0}^{T} \E[\L_{OU}^x\Phi(t,X^\epsilon(t),m^\epsilon(t))] dt,
     \end{align*}
 by definition of the auxiliary function $\Phi$, since $\psi_b$ and $\psi_\sigma$ are solutions of Poisson equations.

 On the other hand, applying It\^o's formula also gives the identity
 \begin{multline*}
     \E[\Phi(T,X^\epsilon(T),m^\epsilon(T))] - \E[\Phi(0,X^\epsilon(0),m^\epsilon(0))]\\
     = \int_{0}^{T} \E[(\partial_t + b \cdot \nabla_x + \frac{1}{2} \sigma \sigma^*:\nabla_x^2 + \frac{1}{\epsilon}\L_{OU}^x) \Phi(t,X^\epsilon(t),m^\epsilon(t))] dt.
 \end{multline*}
 Combining the two expressions then gives
 \begin{align*}
     \E[\varphi(X^\epsilon(T))] - &\E[\varphi(X(T))] = \int_{0}^{T} \E[\L_{OU}^x\Phi(t,X^\epsilon(t),m^\epsilon(t))] dt\\
     &= \epsilon \left( \E[\Phi(T,X^\epsilon(T),m^\epsilon(T))] - \E[\Phi(0,X^\epsilon(0),m^\epsilon(0))] \right)\\
     &\phantom{=} - \epsilon \int_{0}^{T} \E[(\partial_t + b \cdot \nabla_x + \frac{1}{2} \sigma \sigma^*:\nabla_x^2) \Phi(t,X^\epsilon(t),m^\epsilon(t))] dt.
 \end{align*}

 Using the regularity estimates from Lemma~\ref{lem:Kolmo2} and the identity
 \begin{equation*}
     \partial_t u = \overline b \cdot \nabla_x u + \overline \sigma \, \overline \sigma^* : \nabla_x^2 u,
 \end{equation*}
 it is then straightforward to obtain~\eqref{eq:lem2}. This concludes the proof of Lemma~\ref{lem2}.
 \end{proof}

\begin{proof}[Proof of Lemma~\ref{lem3}]
Let us introduce the continuous-time auxiliary process $\tilde{X}$, such that, for $t \in [t_n,t_{n+1}]$, one has
    \begin{equation*}
        \tilde{X}(t) = X_n + (t-t_n)b(X_n,h(X_n)\gamma_n) + \sigma(X_n,h(X_n)\gamma_n) (B(t) - B(t_n)).
    \end{equation*}
Introduce also the second-order differential operator $\tilde \L_n = b(X_n,h(X_n)\gamma_n) \cdot \nabla + \frac{1}{2} \sigma \sigma^* (X_n,h(X_n)\gamma_n) : \nabla^2$. With this notation, for any function $\phi \in \C^2(\T^d)$, It\^o's formula gives, for all $t\in[t_n,t_{n+1}]$,
    \begin{equation} \label{eq:Ito_X-tilde}
        \phi(\tilde X(t)) - \phi(X_n) = \int_{t_n}^t \tilde \L_n \phi(\tilde X(s)) ds + \int_{t_n}^t \nabla \phi (\tilde X(s)) \cdot \sigma(X_n,h(X_n)\gamma_n) dB_s.
    \end{equation}
Using the same (standard) arguments as in the proof of Lemma~\ref{lem1}, one obtains the following decomposition of the error:
    \begin{align*}
        \E[\varphi(&X_N)] - \E[\varphi(X(T))] = \E[u(0,X_N)] - \E[\varphi(T,X_0)]\\
        &= \sum_{n=0}^{N-1} \int_{t_n}^{t_{n+1}} \E[u(T-t_{n+1},X_{n+1}) - u(T-t_n,X_n)]\\
        &= \sum_{n=0}^{N-1} \int_{t_n}^{t_{n+1}} \E[(-\partial_t + \tilde \L_n) u(T-t,\tilde X(t))] dt\\
        &= \sum_{n=0}^{N-1} \int_{t_n}^{t_{n+1}} \E\left[\left(b(X_n,h(X_n)\gamma_n) - \overline{b}(\tilde X(t))\right) \cdot \nabla_x u(T-t,\tilde X(t))\right]dt\\
        &\phantom{=} + \frac 1 2 \sum_{n=0}^{N-1} \int_{t_n}^{t_{n+1}} \E\left[\left(\sigma \sigma^* (X_n,h(X_n)\gamma_n) - \overline \sigma \, \overline \sigma^* (\tilde X(t))\right) : \nabla_x^2 u(T-t,\tilde X(t)) \right]dt.
    \end{align*}
The error term $\E\left[\left(b(X_n,h(X_n)\gamma_n) - \overline{b}(\tilde X(t))\right) \cdot \nabla_x u(T-t,\tilde X(t))\right]$ (with $V=b$ and $U=\nabla_xu$) and $\E\left[\left(\sigma \sigma^* (X_n,h(X_n)\gamma_n) - \overline \sigma \, \overline \sigma^* (\tilde X(t))\right) : \nabla_x^2 u(T-t,\tilde X(t)) \right]$ (with $V=\sigma \sigma^*$ and $U=\nabla_x^2 u$) are written as
    \begin{equation*}
        \left( V(X_n,h(X_n) \gamma_n) - \overline V(\tilde X(t)) \right) U(T-t, \tilde X(t)).
    \end{equation*}
Note that $\tilde X(t_n) = X_n$. As a consequence,
\[
\E\left[\left( V(X_n,h(X_n) \gamma_n) - \overline V(\tilde X(t)) \right) U(T-t, \tilde X(t))\right]=\delta_n^1(t)+\delta_n^2(t)+\delta_n^3(t)+\delta_n^4(t),
\]
where
\begin{align*}
\delta_n^1(t)&=\E \left[ \left( V(X_n,h(X_n)\gamma_n) - \overline V(X_n) \right) U(T-t,X_n) \right]\\
\delta_n^2(t)&=\E \left[ \left( \overline V(X_n) - \overline V(\tilde X(t)) \right) U(T-t,X_n)\right]\\
\delta_n^3(t)&=\E \left[ \left( V(X_n,h(X_n)\gamma_n) - \overline V(X_n) \right) \left( U(T-t,\tilde X(t)) - U(T-t,X_n) \right)\right]\\
\delta_n^4(t)&=\E \left[  \left( \overline V(X_n) - \overline V(\tilde X(t)) \right) \left( U(T-t,\tilde X(t)) - U(T-t,X_n) \right)\right].
\end{align*}

It remains to treat the four error terms $\delta_n^j(t)$, $j=1,2,3,4$.

Let us start with the most important observation: by definition of $\overline{V}(x)=\int V(x,m)d\nu^x(m)$, the independence of the random variables $X_n$ and $\gamma_n$  yields the identity
\begin{multline*}
    \E \left[ \left( V(X_n,h(X_n)\gamma_n) - \overline V(X_n) \right) U(T-t,X_n) \right]\\
    = \E \left[ \E \left[ \left( V(X_n,h(X_n)\gamma_n) - \overline V(X_n) \right) U(T-t,X_n) \mid X_n \right] \right] = 0.
\end{multline*}
The fact that this term vanishes is fundamental since it justifies the consistency of the scheme~\eqref{eq:scheme-limit-av} with the limiting equation~\eqref{eq:limit_SDE-av} (see also Theorem~\ref{th:AP-av} and its proof), and the AP property.

To treat the second term, observe that $X_n$ and $(\tilde X(s),B(s))_{s > t_n}$ are independent random variables, thus conditioning with respect to $X_n$ and applying It\^o's formula~\eqref{eq:Ito_X-tilde} gives
\begin{multline*}
    \E \left[ \left( \overline V(X_n) - \overline V(\tilde X(t)) \right) U(T-t,X_n) \right] = \E \left[ \int_{t_n}^t \tilde \L_n \overline V(\tilde X(s)) ds\, U(T-t,X_n) \right].
\end{multline*}
Since $U$ is bounded (owing to the regularity estimates from Lemma~\ref{lem:Kolmo2}) and since $\overline V \in \C^2(\T^d)$ (by assumptions on the coefficients $b$ and $\sigma$, see Assumption~\ref{ass:coeffs-av}), one obtains
\begin{equation*}
    \abs{\E \left[ \left( \overline V(X_n) - \overline V(\tilde X(t)) \right) U(T-t,X_n) \right]} \lesssim t-t_n.
\end{equation*}

The treatment of the third term uses a conditioning argument, and It\^o's formula~\eqref{eq:Ito_X-tilde}: one has
\begin{multline*}
    \E \left[ \left( V(X_n,h(X_n)\gamma_n) - \overline V(X_n) \right) \left( U(T-t,\tilde X(t)) - U(T-t,X_n) \right) \mid X_n,\gamma_n \right]\\
    = \E \left[ \left( V(X_n,h(X_n)\gamma_n) - \overline V(X_n) \right) \int_{t_n}^t \tilde \L_n U(T-t,\tilde X(s)) ds \mid X_n,\gamma_n \right].
\end{multline*}
Using the regularity properties from Lemma~\ref{lem:Kolmo2}, one obtains
\begin{equation*}
    \abs{\E \left[ \left( V(X_n,h(X_n)\gamma_n) - \overline V(X_n) \right) \left( U(T-t,\tilde X(t)) - U(T-t,X_n) \right) \right]} \lesssim t-t_n.
\end{equation*}

The treatment of the fourth error term is straightforward: since $U$ and $\overline V$ are Lipschitz continuous (owing to Lemma~\ref{lem:Kolmo2}), one has
\begin{align*}
    \abs{\E \left[ \left( U(T-t,\tilde X(t)) - U(T-t,X_n) \right) \left( \overline V(X_n) - \overline V(\tilde X(t)) \right) \right]} &\lesssim \E\left[\norm{\tilde X(t) - X_n}^2\right]\\
    &\lesssim t-t_n.
\end{align*}

The estimates above are of the type
\[
\abs{\delta_n^j(t)}\lesssim t-t_n,
\]
for all $t\in[t_n,t_{n+1}]$ and $j=1,2,3,4$.
Finally, one obtains
    \begin{equation*}
        \abs{\E[\varphi(X_N)] - \varphi(X(T))} \lesssim \sum_{n=0}^{N-1} \int_{t_n}^{t_{n+1}} (t-t_n) dt \lesssim \Delta t,
    \end{equation*}
    which concludes the proof of Lemma~\ref{lem3}.
\end{proof}

\begin{proof}[Proof of Lemma~\ref{lem4}]
    The idea is to adapt the proof of Lemma~\ref{lem2} (see Section~\ref{sec:app-av}) to the discrete-time situation. Let us start with preparatory computations. A telescoping sum argument yields the equality
    \begin{align}
        \E[\varphi(X_N^\epsilon)] - \E[\varphi(X_N)] &= \E[u_0(X_N^\epsilon)] - u_N(X_0^\epsilon) \nonumber\\
        &= \sum_{n=0}^{N-1}\left(\E[u_{N-n-1}(X_{n+1}^\epsilon)] - \E[u_{N-n}(X_{n}^\epsilon)]\right), \label{eq:averaging_error_discrete-error}
    \end{align}
    where the auxiliary function $u_n$ is defined in Lemma~\ref{lem:Kolmo3}. Using the definition of the scheme~\eqref{eq:scheme-AP-av}, and Markov property combined with the expression of the limiting scheme~\eqref{eq:scheme-limit-av}, one obtains
    \begin{align*}
        u_{N-n-1}(X_{n+1}^\epsilon) &= u_{N-n-1} \left( X_{n}^\epsilon + \Delta t b(X_n^\epsilon,m^\epsilon_{n+1}) + \sqrt{\Delta t} \sigma(X_n^\epsilon,m^\epsilon_{n+1}) \Gamma_n \right),\\
        \E[u_{N-n}(X_{n}^\epsilon)] &= \E\left[u_{N-n-1} \left( X_{n}^\epsilon + \Delta t b(X_n^\epsilon,h(X^\epsilon_n)\gamma_n) + \sqrt{\Delta t} \sigma(X_n^\epsilon,h(X^\epsilon_n)\gamma_n) \Gamma_n \right)\right].
    \end{align*}

A second order Taylor expansion then gives
    \begin{align*}
        \E[&u_{N-n-1}(X_{n+1}^\epsilon)] - \E[u_{N-n}(X_{n}^\epsilon)] \\
        &=\Delta t \E\left[\left(b(X_n^\epsilon,m^\epsilon_{n+1}) - b(X_n^\epsilon,h(X^\epsilon_n)\gamma_n)\right) \cdot \nabla_x u_{N-n-1}(X^\epsilon_n)\right]\\
        &\phantom{=} + \frac{\Delta t}{2} \E\left[\left(\sigma \Gamma_n \Gamma_n^* \sigma^* (X_n^\epsilon,m^\epsilon_{n+1}) - \sigma \Gamma_n \Gamma_n^* \sigma^* (X_n^\epsilon,h(X^\epsilon_n)\gamma_n)\right) : \nabla_x^2 u_{N-n-1}(X^\epsilon_n)\right]\\
        &\phantom{=}+ \sqrt{\Delta t} \E\left[\left( \sigma(X_n^\epsilon,m^\epsilon_{n+1}) - \sigma(X_n^\epsilon,h(X^\epsilon_n)\gamma_n) \right) \Gamma_n \cdot \nabla_x u_{N-n-1}(X^\epsilon_n)\right]\\
        &\phantom{=} + \Delta t^{3/2} \E\left[\left(\sigma \Gamma_n b^* (X_n^\epsilon,m^\epsilon_{n+1}) - \sigma \Gamma_n b^* (X_n^\epsilon,h(X^\epsilon_n)\gamma_n)\right) : \nabla_x^2 u_{N-n-1}(X^\epsilon_n)\right]\\
        &\phantom{=} + \Delta t^2 \E[R_n(\Delta t)].
    \end{align*}
Using the regularity estimates from Lemma~\ref{lem:Kolmo3}, one has $\abs{R_n(\Delta t)} \lesssim 1$. Note that the terms of the orders $\Delta t^{\frac12}$ and $\Delta t^{\frac32}$ in the right-hand side above vanish, since the random variables $\Gamma_n$ and $(X^\epsilon_n, m^\epsilon_{n+1}, \gamma_n)$ are independent, and $\E[\Gamma_n]=0$. In addition, since $\E[\Gamma_n\Gamma_n^*]=I$, a conditioning argument yields
    \begin{equation} \label{eq:averaging_error_discrete-error2}
    \begin{aligned}
        \E[u_{N-n-1}&(X_{n+1}^\epsilon) - u_{N-n}(X_{n}^\epsilon)]\\
        &= \Delta t \E \left[ \left(b(X_n^\epsilon,m^\epsilon_{n+1}) - b(X_n^\epsilon,h(X^\epsilon_n)\gamma_n) \right) \cdot \nabla_x u_{N-n-1}(X^\epsilon_n) \right]\\
        &\phantom{=} + \Delta t \E \left[ \left(\sigma \sigma^* (X_n^\epsilon,m^\epsilon_{n+1}) - \sigma \sigma^* (X_n^\epsilon,h(X^\epsilon_n)\gamma_n)\right) : \nabla_x^2 u_{N-n-1}(X^\epsilon_n) \right]\\
        &\phantom{=} + \Delta t^2 \E[R_n(\Delta t)].
    \end{aligned}
    \end{equation}
Like in the proofs of Theorem~\ref{th:AP-av} and of Lemma~\ref{lem3}, a conditioning argument allows us to rewrite the expressions above in terms of the functions $\overline{b}$ and $\overline{\sigma}$: one has
    \begin{multline*}
        \E[\left(b(X_n^\epsilon,m^\epsilon_{n+1}) - b(X_n^\epsilon,h(X^\epsilon_n)\gamma_n\right) \cdot \nabla_x u_{N-n-1}(X^\epsilon_n)]\\
         = \E[\left(b(X_n^\epsilon,m^\epsilon_{n+1}) - \overline{b}(X_n^\epsilon\right) \cdot \nabla_x u_{N-n-1}(X^\epsilon_n)],
    \end{multline*}
    and
    \begin{multline*}
        \E[\left(\sigma \sigma^* (X_n^\epsilon,m^\epsilon_{n+1}) - \sigma \sigma^* (X_n^\epsilon,h(X^\epsilon_n)\gamma_n\right) : \nabla_x^2 u_{N-n-1}(X^\epsilon_n)]\\
        = \E[\left(\sigma \sigma^* (X_n^\epsilon,m^\epsilon_{n+1}) - \overline \sigma \, \overline \sigma^* (X_n^\epsilon\right) : \nabla_x^2 u_{N-n-1}(X^\epsilon_n)],
    \end{multline*}
We are now in position to employ similar arguments as in the proof of Lemma~\ref{lem2} (see Section~\ref{sec:app-av}), with important modifications due to the discrete-time setting. Introduce the auxiliary parameter $\tau = \frac{\Delta t}{\epsilon}$. Instead of studying Poisson equations associated with the infinitesimal generator $\L_{OU}^x$, one needs to consider the generator $L_\tau^x$  and the transition semigroup of a Markov chain: let
    \begin{equation*}
        L_\tau^x = \frac{P_\tau^x-I}\tau \mbox{ where } P_\tau^x \phi(m) \doteq \E_{\gamma \sim \mathcal N(0,1)} [\phi(e^{-\tau}m + \sqrt{1 - e^{-2 \tau}} h(x) \gamma)].
    \end{equation*}
We claim that the function $\psi_b^\tau$ defined by
    \begin{equation*}
        \psi_b^\tau(x,m) = - \tau \sum_{n=0}^{+\infty} (\left( P_\tau^x \right)^n b(x,m) - \overline{b}(x)).
    \end{equation*}
is well-defined and solves the Poisson equation $L_\tau^x \psi_b^\tau(x,\cdot) = b(x,\cdot) - \overline{b}(x)$. Indeed, let $\left( m^x_n(m) \right)_n$ be defined by
    \begin{equation*}
m^x_{n+1}(m) = e^{-\tau} m^x_n(m) + \sqrt{1 - e^{-2 \tau}} h(x) \gamma_n,\quad m^x_0(m) = m.
    \end{equation*}
Then, for all $m$ and $m' \in \R$, and all $n\in\N$, one has
    \begin{equation*}
        m^x_{n}(m) - m^x_{n}(m') = e^{-n\tau}(m-m').
    \end{equation*}
Observe that since $b$ is a Lipschitz continuous function, standard arguments give the following upper bound: for all $n\in\N$, $x\in\T^d$ and all $m,m'\in\R$, if $\delta_n(x,m) = \left( P_\tau^x \right)^n b(x,m) - \overline b(x)$, then one has
    \begin{align*}
        \big|\delta_n(x,m)\big|&=\big|\left( P_\tau^x \right)^n b(x,m) - \int \left( P_\tau^x \right)^n b(x,m')d\nu^x(m')\big|\\
        &\le \int \big|\E b(x,m^x_n(m))-\E b(x,m^x_n(m')\big|d\nu^x(m')\\
        &\lesssim e^{-n\tau}(1+|m|).
    \end{align*}
Similarly, since the derivatives of $m^x_n$ with respect to $x$ do not depend on $m$, one can check that the inequality above holds for $D_x \delta$ and $D_x^2 \delta$ then concludes the proof of the claim. Since $\tau\sum_{n=0}^{\infty}e^{-n\tau}=\frac{\tau}{1-e^{-\tau}}\le \max(\tau,1)$, for all $\tau\in(0,\infty)$, one obtains inequalities of the type
    \begin{equation} \label{eq:estimate_psi}
        \frac{\|\psi(x,m)\|}{1+|m|} \lesssim \max(\tau,1),
    \end{equation}
for $\psi_b^\tau$, and its derivatives $D_x \psi_b^\tau$ and $D_x^2 \psi_b^\tau$.

Similarly, define for all $x\in\T^d$ and $m\in\R$,
    \begin{equation*}
        \psi_\sigma^\tau(x,m) = - \tau \sum_{n=0}^{+\infty} (\left( P_\tau^x \right)^n\sigma \sigma^* (x,m) - \overline \sigma \, \overline \sigma^* (x)).
    \end{equation*}
Then $\psi_\sigma^\tau$ is well-defined and solves the Poisson equation $L_\tau^x\psi_\sigma(x,\cdot)=\sigma\sigma^\star(x,\cdot)-\overline{\sigma}(x)\overline{\sigma}^\star(x)$, by definition of $\overline{\sigma}(x)$, see Assumption~\ref{ass:av}. In addition, $\psi_\sigma^\tau$ and its derivatives satisfy upper bound of the type~\eqref{eq:estimate_psi}.

Like in the proof of Lemma~\ref{lem2} (see Section~\ref{sec:app-av}), introduce the auxiliary function defined by
    \begin{equation*}
        \Phi_n(x,m) = \psi_b^\tau(x,m) \cdot \nabla_x u_{N-n-1}(x) + \psi_\sigma^\tau(x,m) : \nabla_x^2 u_{N-n-1}(x),
    \end{equation*}
for all $n\in\{0,\ldots,N-1\}$, $x\in\T^d$ and $m\in\R$, where $u_n$ is defined in Lemma~\ref{lem:Kolmo3}. Combining the decomposition of the error~\eqref{eq:averaging_error_discrete-error} and the identity~\eqref{eq:averaging_error_discrete-error2}, one then obtains the following new expression for the error:
    \begin{align}
        \E[\varphi(X_N^\epsilon)] - \E[\varphi(X_N)]
        &= \Delta t^2 \sum_{n=0}^{N-1} \E[R_n(\Delta t)] + \Delta t \sum_{n=0}^{N-1} L_\tau^{X^\epsilon_n} \Phi_n(X^\epsilon_n,m^\epsilon_{n+1}). \label{eq:averaging_error_discrete-error3}
    \end{align}
On the one hand, a telescoping sum argument yields the following expression:
    \begin{align*}
        \E\bigl[ P_\tau^{X^\epsilon_N} &\Phi_N(X^\epsilon_N,m^\epsilon_N) - P_\tau^{X^\epsilon_0} \Phi_0(X^\epsilon_0,m^\epsilon_0) \bigr] \\
        &= \sum_{n=0}^{N-1} \E\left[ P_\tau^{X^\epsilon_{n+1}} \Phi_{n+1}(X^\epsilon_{n+1},m^\epsilon_{n+1}) - P_\tau^{X^\epsilon_n} \Phi_n(X^\epsilon_n,m^\epsilon_n) \right]\\
        &= \sum_{n=0}^{N-1} \E\left[ P_\tau^{X^\epsilon_{n+1}} \Phi_{n+1}(X^\epsilon_{n+1},m^\epsilon_{n+1}) - P_\tau^{X^\epsilon_{n+1}} \Phi_n(X^\epsilon_{n+1},m^\epsilon_{n+1}) \right]\\
        &\phantom{=} + \sum_{n=0}^{N-1} \E\left[ P_\tau^{X^\epsilon_{n+1}} \Phi_n(X^\epsilon_{n+1},m^\epsilon_{n+1}) - P_\tau^{X^\epsilon_n} \Phi_n(X^\epsilon_n,m^\epsilon_{n+1}) \right]\\
        &\phantom{=} + \sum_{n=0}^{N-1} \E\left[ P_\tau^{X^\epsilon_n} \Phi_n(X^\epsilon_n,m^\epsilon_{n+1}) - P_\tau^{X^\epsilon_n} \Phi_n(X^\epsilon_n,m^\epsilon_n) \right].
    \end{align*}
Note that using Markov property, the first sum on the right-hand side above can be written as $\sum_{n=0}^{N-1} \E\left[ \Phi_{n+1}(X^\epsilon_{n+1},m^\epsilon_{n+2}) - \Phi_n(X^\epsilon_{n+1},m^\epsilon_{n+2}) \right]$.

On the other hand, by definition of the operator $L_\tau^x$ with the parameter $\tau=\frac{\Delta t}{\epsilon}$, one obtains
    \begin{align*}
        \Delta t \E[L_\tau^{X^\epsilon_n} \Phi_n(X^\epsilon_n,m^\epsilon_{n+1})] &= \epsilon \E[P_\tau^{X^\epsilon_n} \Phi_n(X^\epsilon_n,m^\epsilon_{n+1}) - \Phi_n(X^\epsilon_n,m^\epsilon_{n+1})]\\
        &= \epsilon \E[P_\tau^{X^\epsilon_n} \Phi_n(X^\epsilon_n,m^\epsilon_{n+1}) - P_\tau^{X^\epsilon_n} \Phi_n(X^\epsilon_n,m^\epsilon_n)].
    \end{align*}
Finally, combining the two identities above, one obtains the following expression for the error:
    \begin{align*}
        \E[\varphi(X_N^\epsilon)] - \E[\varphi(X_N)] &= \Delta t^2 \sum_{n = 0}^{N-1} \E[R_n]\\
        &\phantom{=} + \epsilon \left( \E\left[ \Phi_N(X^\epsilon_N,m^\epsilon_{N+1}) - \Phi_0(X^\epsilon_0,m^\epsilon_1)\right] \right)\\
        &\phantom{=} - \epsilon \sum_{n=0}^{N-1} \E\left[ \Phi_{n+1}(X^\epsilon_{n+1},m^\epsilon_{n+2}) - \Phi_n(X^\epsilon_{n+1},m^\epsilon_{n+2}) \right]\\
        &\phantom{=} - \epsilon \sum_{n=0}^{N-1} \E\left[ P_\tau^{X^\epsilon_{n+1}} \Phi_n(X^\epsilon_{n+1},m^\epsilon_{n+1}) - P_\tau^{X^\epsilon_n} \Phi_n(X^\epsilon_n,m^\epsilon_{n+1}) \right].
    \end{align*}
It then remains to use auxiliary upper bounds to deduce the result, in particular using~\eqref{eq:estimate_psi}. Note that $\epsilon\max(\tau,1)=\max(\Delta t,\epsilon)$.
    \begin{itemize}
        \item as explained above, $\E[\abs{R_n(\Delta t)}] \lesssim 1$, thus the first term satisfies
        \begin{equation*}
            \E[\abs{\Delta t^2 \sum_{n = 0}^{N-1} R_n}] \lesssim \Delta t.
        \end{equation*}

        \item Using the upper bound~\eqref{eq:Kolmo3-1} from Lemma~\ref{lem:Kolmo3} and~\eqref{eq:estimate_psi} with $\psi=\psi_b^\tau$ and $\psi=\psi_\sigma^\tau$, the second term satisfies
       \begin{equation*}
           \epsilon \left( \E\left[ \Phi_N(X^\epsilon_N,m^\epsilon_{N+1}) - \Phi_0(X^\epsilon_0,m^\epsilon_1)\right] \right) \lesssim \max(\Delta t,\epsilon).
       \end{equation*}

        \item Observe that one has for all $x\in\T^d$ and $x\in\R$,
        \begin{align*}
            \Phi_{n+1}(x,m) - \Phi_n(x,m) &= \psi_b^\tau(x,m) \cdot \left(\nabla_x u_{N-n}(x) - \nabla_x u_{N-n-1}(x)\right)\\
            &\phantom{=} + \psi_\sigma^\tau(x,m) : \left(\nabla_x^2 u_{N-n}(x) - \nabla_x^2 u_{N-n-1}(x)\right).
        \end{align*}
        Using the upper bound~\eqref{eq:Kolmo3-2} from Lemma~\ref{lem:Kolmo3} and~\eqref{eq:estimate_psi}, one has
        \begin{equation*}
            \abs{\Phi_{n+1}(x,m) - \Phi_n(x,m)} \lesssim \Delta t\max(\tau,1)(1+|m|).
        \end{equation*}
As a consequence, owing to~\eqref{eq:moment_m-eps-n} and Assumption~\ref{ass:init}, the third term satisfies
        \begin{align*}
            \epsilon \sum_{n=0}^{N-1} \E\left[ \Phi_{n+1}(X^\epsilon_{n+1},m^\epsilon_{n+2}) - \Phi_n(X^\epsilon_{n+1},m^\epsilon_{n+2}) \right]&\lesssim \epsilon \sum_{n=0}^{N-1} \Delta t \max(\tau,1)
            \\&\lesssim \max(\Delta t,\epsilon).
        \end{align*}

        \item Note that $\Phi_n$ and its derivatives satisfy the upper bound~\eqref{eq:estimate_psi}, owing to~\eqref{eq:Kolmo3-1} from Lemma~\ref{lem:Kolmo3}. Let
        \begin{equation*}
            f_n^\tau(x,m) = P_\tau^x \Phi_n(x,m) = \E_{\gamma \sim \mathcal N(0,1)} [\Phi_n(x,e^{-\tau}m + \sqrt{1 - e^{-2 \tau}} h(x) \gamma)].
        \end{equation*}
It is straightforward to check that $f_n^\tau$ is twice differentiable. In addition, $f_n^\tau$ and its derivatives satisfy~\eqref{eq:estimate_psi}. Using a second order Taylor expansion, one obtains
\begin{align*}
\abs{P_\tau^{X^\epsilon_{n+1}} \Phi_n(X^\epsilon_{n+1},m^\epsilon_{n+1}) &- P_\tau^{X^\epsilon_n} \Phi_n(X^\epsilon_n,m^\epsilon_{n+1})}\\
&=\abs{f_n^\tau(X^\epsilon_{n+1},m_{n+1}^{\epsilon}) - f_n^\tau(X^\epsilon_n,m_n^\epsilon)}\\
&\lesssim \Delta t\max(\tau,1) (1 + \abs{m}).
\end{align*}
Finally, using~\eqref{eq:moment_m-eps-n} and Assumption~\ref{ass:init}, one obtains
       \begin{equation*}
           \epsilon \sum_{n=0}^{N-1} \E\left[ P_\tau^{X^\epsilon_{n+1}} \Phi_n(X^\epsilon_{n+1},m^\epsilon_{n+1}) - P_\tau^{X^\epsilon_n} \Phi_n(X^\epsilon_n,m^\epsilon_{n+1}) \right] \lesssim\max(\Delta t,\epsilon).
       \end{equation*}
    \end{itemize}
Gathering the estimates then concludes the proof of Lemma~\ref{lem4}.
\end{proof}

We refer to~\cite[Theorem 1.3.6]{Cerrai} for the proof of Lemma~\ref{lem:Kolmo2}. It thus remains to provide the arguments for the proofs of Lemmas~\ref{lem:Kolmo1} and~\ref{lem:Kolmo3}. The strategy is standard in the literature, see for instance~\cite{Cerrai}. As a consequence, in order to reduce the length of the manuscript, below the details are only given for the proofs of the estimates for first-order derivatives.

\begin{proof}[Proof of Lemma~\ref{lem:Kolmo1}]
Owing to~\cite[Proposition~1.3.5]{Cerrai}, for all ${\underline k} = (k_x,k_m) \in \R^d \times \R$, one has
    \begin{equation*}
        D_{x,m} u^\epsilon(t,x,m) \cdot {\underline k} = \E_{X^\epsilon(0) = x, m^\epsilon(0) = m} [D \varphi(X^\epsilon(t)) \cdot \eta^{{\epsilon,\underline k}}_x(t)],
    \end{equation*}
where the process $\eta^{{\epsilon,\underline k}} = (\eta^{{\epsilon,\underline k}}_x, \eta^{{\epsilon,\underline k}}_m)$ is the solution of the first variation equation associated with~\eqref{eq:SDE-av}: for all $t\ge 0$,
\begin{equation*}
    \left \lbrace
    \begin{aligned}
        d \eta^{{\epsilon,\underline k}}_{x,t} &= D b (X^\epsilon_t, m^\epsilon_t) \cdot \eta^{{\epsilon,\underline k}}_t dt + D \sigma (X^\epsilon_t, m^\epsilon_t) \cdot \eta^{{\epsilon,\underline k}}_t dB_t\\
        d \eta^{{\epsilon,\underline k}}_{m,t} &= -\frac{\eta^{{\epsilon,\underline k}}_{m,t}}{\epsilon} dt + \frac{\sqrt{2} D h(X^\epsilon_t) \cdot \eta^{{\epsilon,\underline k}}_{x,t}}{\sqrt{\epsilon}} d\beta_t,
    \end{aligned}
    \right.
\end{equation*}
with initial conditions $\eta^{{\epsilon,\underline k}}_x(0) = k_x$ and $\eta^{{\epsilon,\underline k}}_m(0) = k_m$.

On the one hand, the component $\eta^{{\epsilon,\underline k}}_m$ satisfies the following equality,
    \begin{equation*}
        \eta^{{\epsilon,\underline k}}_m(t) = e^{-\frac{t}{\epsilon}} k_m + \frac{1}{\sqrt{\epsilon}} \int_0^t e^{-\frac{t-s}{\epsilon}} \sqrt{2} D h(X^\epsilon(s)) \cdot \eta^{{\epsilon,\underline k}}_x(s) dB_s,
    \end{equation*}
and by means of It\^o's isometry formula, one has
    \begin{equation*}
        \E[\norm{\eta^{{\epsilon,\underline k}}_m(t)}^2] \leq k_m^2 + \frac{2 \norm{h}_{\C^1(\T^d)}}{\epsilon} \int_0^t e^{-\frac{2(t-s)}{\epsilon}} \E[\norm{\eta^{{\epsilon,\underline k}}_x(s)}^2] ds.
    \end{equation*}
On the other hand, the functions $b$ and $\sigma$ are globally Lipschitz continuous (see Assumption~\ref{ass:coeffs-av}), and one has $\norm{\eta^{{\epsilon,\underline k}}}^2 = \norm{\eta^{{\epsilon,\underline k}}_x}^2 + \norm{\eta^{{\epsilon,\underline k}}_m}^2$ ; using It\^o's isometry formula, and Minkowski's and Young's inequalities, one obtains
    \begin{equation*}
        \E[\norm{\eta^{{\epsilon,\underline k}}_x(t)}^2] \lesssim \norm{k_x}^2 + \int_0^t \left( \E[\norm{\eta^{{\epsilon,\underline k}}_x(s)}^2] + \E[\norm{\eta^{{\epsilon,\underline k}}_m(s)}^2] \right) ds.
    \end{equation*}
Combining the two estimates above then yields
    \begin{align*}
        \sup_{s \in [0,t]} \E[\norm{\eta^{{\epsilon,\underline k}}_x(s)}^2] &\lesssim \norm{{\underline k}}^2 + \int_0^t \left( \E[\norm{\eta^{{\epsilon,\underline k}}_x(s)}^2] + \frac{1}{\epsilon} \int_0^s e^{-2\frac{s-r}{\epsilon}} \E[\norm{\eta^{{\epsilon,\underline k}}_x(r)}^2] dr \right) ds\\
        &\lesssim \norm{{\underline k}}^2 + \int_0^t \sup_{r \in [0,s]} \E[\norm{\eta^{{\epsilon,\underline k}}_x(r)}^2] ds,
    \end{align*}
    where the inequality $\frac{1}{\epsilon} \int_0^s e^{-2\frac{s-r}{\epsilon}} dr \leq \frac 1 2$ has been used. Applying Gronwall's lemma, then inserting the result in the estimate above, one obtains the upper bounds
    \begin{equation*}
        \sup_{t \in [0,T]} \E[\norm{\eta^{{\epsilon,\underline k}}_x(t)}^2] \lesssim \norm{{\underline k}}^2~,\quad \sup_{t \in [0,T]} \E[\norm{\eta^{{\epsilon,\underline k}}_m(t)}^2] \lesssim \norm{{\underline k}}^2.
    \end{equation*}
Since $\varphi$ is Lipschitz continuous, using the expression for $D_{x,m}u^\epsilon(t,x,m)$ stated above, one finally obtains~\eqref{eq:Kolmo1} for the first-order derivative: for all ${\underline k}\in\R^{d+1}$, one has
\[
\sup_{(t,x,m) \in [0,T] \times \T^d \times \R} \big|D_{x,m} u^\epsilon(t,x,m) \cdot {\underline k}\big|\lesssim \norm{{\underline k}}.
\]
The treatment of higher-order derivatives follows from similar arguments which are omitted. This concludes the proof of Lemma~\ref{lem:Kolmo1}.
\end{proof}

%

\begin{proof}[Proof of Lemma~\ref{lem:Kolmo3}]
For all $k \in \R^d$, one has
\[
D u_n (x) \cdot k = \E_x[D \varphi(X_n) \cdot \eta_n^k],
\]
where $\eta_0^h=h$, and for all $n\in\{0,\ldots,N-1\}$, one has
    \begin{align*}
        \eta_{n+1}^k &= \eta_n^k + \Delta t D_x b(X_n,h(X_n)\gamma_n) \cdot \eta_n^k\\
        &\phantom{=} + \Delta t \partial_m b(X_n,h(X_n)\gamma_n) D h(X_n) \cdot \eta_n^k \gamma_n\\
        &\phantom{=} + \sqrt{\Delta t} D_x \sigma (X_n,h(X_n)\gamma_n) \cdot \eta_n^k \Gamma_n\\
        &\phantom{=} + \sqrt{\Delta t} \partial_m \sigma (X_n,h(X_n)\gamma_n) D h(X_n) \cdot \eta_n^k \Gamma_n \gamma_n.
    \end{align*}
The functions $b$, $h$ and $\sigma$ are Lipschitz continuous (see Assumption~\ref{ass:coeffs-av}). Since $\gamma_n$ and $\Gamma$ are independent centered Gaussian random variables, it is straightforward to obtain the upper bound
    \begin{equation*}
        \E[\norm{\eta_{n+1}^{k}}^2] \lesssim (1 + \Delta t)\E[\norm{\eta_n^k}^2] .
    \end{equation*}
A straightforward recursion argument then gives, for all $n\in\{0,\ldots,N\}$,
\begin{equation*}
\E[\norm{\eta_n^k}^2]\lesssim \|k\|^2
\end{equation*}
and one obtains~\eqref{eq:Kolmo3-1} for the first-order derivative: for all $n\in\{0,\ldots,N\}$,
\begin{equation*}
|D u_n (x) \cdot k|\lesssim \|k\|.
\end{equation*}
The treatment of higher-order derivatives would be similar and is omitted. It thus remains to prove~\eqref{eq:Kolmo3-2}. On the one hand, by definition~\eqref{eq:scheme-limit-av}, a second order Taylor expansion yields, for all $x\in\T^d$ and $n\in\{0,\ldots,N-1\}$,
    \begin{equation*}
        \abs{u_{n+1}(x)-u_n(x)} = \abs{\E_x[\varphi(X_{n+1}) - \varphi(X_n)]} \lesssim \Delta t,
    \end{equation*}
On the other hand, one has
    \begin{align*}
        D u_{n+1}(x) \cdot k - D u_n (x) \cdot k &= \E_x[ D \varphi(X_{n+1}) \cdot \eta_{n+1}^k - D \varphi(X_n) \cdot \eta_n^k]\\
        &= \E[ \left( D \varphi(X_{n+1}) - D \varphi(X_n) \right) \cdot (\eta_{n+1}^k - \eta_n^k)]\\
        &\phantom{=} + \E[ D \varphi(X_n) \cdot (\eta_{n+1}^k - \eta_n^k)]\\
        &\phantom{=} + \E[ \left( D \varphi(X_{n+1}) - D \varphi(X_n) \right) \cdot \eta_n^k].
    \end{align*}
It is straightforward to check that one has the inequalities $\E[\|X_{n+1}-X_n\|^2]\lesssim \Delta t$ and $\E[\|\eta_{n+1}^k-\eta_n^k\|^2]\lesssim \Delta t\|k\|^2$. Since $\varphi$ is Lipschitz continuous, using Cauchy-Schwarz inequality gives
    \begin{equation*}
        \abs{\E[ \left( D \varphi(X_{n+1}(x)) - D \varphi(X_n(x)) \right) \cdot (D X_{n+1}(x) \cdot k - D X_n(x) \cdot k)]} \lesssim  \Delta t\|k\|.
    \end{equation*}
Using a conditional expectation argument, since $\gamma_n,\Gamma_n,X_n$ are independent random variables, one has
\begin{equation*}
\abs{\E[ D \varphi(X_n) \cdot (\eta_{n+1}^k - \eta_n^k)]}=\Delta t\abs{\E[ D \varphi(X_n) \cdot  D_x b(X_n,h(X_n)\gamma_n) \cdot \eta_n^k]}\lesssim \Delta t\|k\|.
\end{equation*}
Finally, using a second-order Taylor expansion and conditioning arguments, one obtains
\begin{equation*}
\abs{\E[ \left( D \varphi(X_{n+1}) - D \varphi(X_n) \right) \cdot \eta_n^k]}\lesssim \Delta t\|k\|.
\end{equation*}
As a consequence, one obtains~\eqref{eq:Kolmo3-2} when $j=0$ and $j=1$.

This concludes the proof of Lemma~\ref{lem:Kolmo3}.
\end{proof}

\section{Conclusion}\label{sec:conclusion}

In this article, we have studied a general notion of Asymptotic Preserving schemes, related to convergence in distribution, for a class of SDE systems in averaging and diffusion approximation regimes. Let us mention that some assumptions made to simplify the setting (the slow component takes values in a compact set $\T^d$ and the fast component is one-dimensional) may easily be relaxed. Note that when the slow component takes values in $\R^d$, it is necessary to also study the stability of the numerical schemes, for instance in mean-square sense.

A limitation of our study is the fact that the fast component is an Ornstein-Uhlenbeck process (when the slow component is frozen): even if the general theory of AP schemes described in Section~\ref{sec:num-gen} holds in more general settings, the construction of implementable AP schemes (such as the ones described in Sections~\ref{sec:num-av} and~\ref{sec:num-diff}) is not straightforward if for instance the fast component is solution of a general ergodic SDE with nonlinear coefficients.

We have also left open the question of obtaining a version of the error estimates stated in Theorem~\ref{th:UA-av} in the diffusion approximation case. This question will be studied in future works.

Finally, it would be natural to apply the recipes for the design of AP schemes described in this article to SPDE models. For instance, in a future work~\cite{BrehierHivertRakotonirina}, we plan to design, analyze and test AP schemes for the stochastic kinetic PDE model considered in~\cite{Rakotonirina}.

\appendix
\section{Derivation of the limiting models}

\subsection{Sketch of proof of Proposition~\ref{propo:limit-SDE-av} (averaging)}\label{sec:app-av}

Let us first give details concerning the construction of the perturbed test function $\varphi^\epsilon$ given by~\eqref{eq:def_phi-epsilon-av}, such that~\eqref{eq:generator_convergence-av} holds. Recall that this construction is used in the statement of Proposition~\ref{propo:AP}.

Owing to the multiscale expansions~\eqref{eq:L_01-av} and~\eqref{eq:def_phi-epsilon-av} of the generator $\L^\epsilon$ and of the perturbed test function $\varphi^\epsilon = \varphi + \epsilon \varphi_1$, one has
\begin{equation} \label{eq:perturbed-test-fn_eqn_expanded-av}
    \L^\epsilon \varphi^\epsilon = \epsilon^{-1} \L_{OU} \varphi +  \left( \L_0 \varphi + \L_{OU} \varphi_1 \right) + \epsilon \L_0 \varphi_1.
\end{equation}

Since the test function $\varphi$ does not depend on $m$, one has $\L_{OU} \varphi = 0$, thus the term of order $\epsilon^{-1}$ in~\eqref{eq:perturbed-test-fn_eqn_expanded-av} vanishes.

Define, for all $x \in \T^d$ and $m \in \R$,
\begin{align*}
    \L \varphi(x) &= \int_\R \L_0 \varphi (x,m) d\nu^x(m)\\
    &= \overline b(x) \cdot \nabla_x \varphi(x) + \overline \sigma \, \overline \sigma^* (x) : \nabla_x^2 \varphi(x)\\
    \vartheta(x,m)&= \L_0 \varphi(x,m) - \L\varphi(x)\\
    &= \left( b(x,m) - \overline b(x) \right) \cdot \nabla_x \varphi(x) + \left( \sigma \sigma^* (x,m) - \overline \sigma \, \overline \sigma^* (x) \right) : \nabla_x^2 \varphi(x),
\end{align*}
where we recall that $\nu^x = \mathcal{N}(0,h(x)^2)$ is the invariant distribution of the ergodic Ornstein-Uhlenbeck process $m^x$ associated to $\L_{OU}$ on $\R$, for any fixed $x\in\mathbb{T}^d$
\begin{equation*}
    d m^x_t = - m^x_t dt + \sqrt{2}h(x) d\beta_t.
\end{equation*}
Let $m^x(t,m)$ denote the solution at time $t$, if the initial condition is given by $m^x(0,m)$. Therefore, the centering condition $\int \vartheta(x,m)d\nu^x(m)=0$ is satisfied and the Poisson equation $-\L_{OU}\varphi_1(x,\cdot)=\vartheta(x,\cdot)$ admits a solution
\begin{equation*}
    \varphi_1(x,m) = \int_0^\infty \E[\vartheta(x,m^x(t,m))] dt.
\end{equation*}

The multiscale expansion~\eqref{eq:perturbed-test-fn_eqn_expanded-av} becomes
\begin{equation*}
    \L^\epsilon \varphi^\epsilon = \L \varphi + \epsilon \L_0 \varphi_1.
\end{equation*}

To prove~\eqref{eq:generator_convergence-av}, it only remains to get estimates on $\L_0 \varphi_1$ uniformly in $\epsilon$. Consider $V = b$ or $V = \sigma \sigma^*$ and let
\begin{equation*}
    \delta(t,x,m) \doteq \E[V(x,m^x(t,m)) - \overline V(x)].
\end{equation*}
Note that for $m$ and $m' \in \R$, one has $m^x(t,m)-m^x(t,m') = (m-m') e^{-t}$. As a consequence, we have
\begin{equation*}
    \norm{\delta(t,x,m) - \delta(t,x,m')} \lesssim \abs{m-m'} e^{-t}.
\end{equation*}
By integrating with respect to $m'$ and using the equality $\overline V(x) = \int V(x,m') d\nu^x(m')$, one obtains
\begin{equation}\label{eq:OU_exponential-mixing-app}
    \norm{\delta(t,x,m)} \lesssim (1+\abs{m}) e^{-t}.
\end{equation}
Since $V$ is of class $\C^3$ with bounded derivatives, and since the derivatives of $m^x(t,m)$ with respect to $x$ do not depend on $m$, it is straightforward to generalize~\ref{eq:OU_exponential-mixing-app} to the derivatives of $\delta$. It gives that $\varphi_1 \in \mathcal C^2 (\T^d \times \R)$ and that $\varphi_1$ and its derivatives have at most linear growth in $m$, hence $\L_0 \varphi_1$ also does. This leads to~\eqref{eq:generator_convergence-av} using~\eqref{eq:moment_m-eps-t}. This concludes the identification of the limiting generator $\L$ using the perturbed test function method. The remaining ingredients of this strategy to prove the convergence in distribution of the process $X^\epsilon$ to the solution $X$ of the limiting equation associated with the limiting generator $\L$ are standard and are thus omitted.

\subsection{Sketch of proof of Proposition~\ref{propo:limit-SDE-diff} (diffusion approximation)}\label{sec:app-diff}

Let us first give details concerning the construction of the perturbed test function $\varphi^\epsilon$ given by~\eqref{eq:def_phi-epsilon-diff}, such that~\eqref{eq:generator_convergence-diff} holds. Recall that this construction is used in the statement of Proposition~\ref{propo:AP}.

Owing to the multiscale expansions~\eqref{eq:L_012-diff} and~\eqref{eq:def_phi-epsilon-diff} of the generator $\L^\epsilon$ and of the perturbed test function $\varphi^\epsilon=\varphi+\epsilon \varphi_1+\epsilon^2\varphi_2$, one has
\begin{multline} \label{eq:perturbed-test-fn_eqn_expanded-diff}
    \L^\epsilon \varphi^\epsilon = \epsilon^{-2} \L_{OU} \varphi + \epsilon^{-1} \left( \L_1 \varphi + \L_{OU} \varphi_1 \right) + \left( \L_0 \varphi + \L_1 \varphi_1 + \L_{OU} \varphi_2 \right) \\+ \epsilon \left( \L_0 \varphi_1 + \L_1 \varphi_2 \right) + \epsilon^2 \L_0 \varphi_2.
\end{multline}

Since the test function $\varphi$ does not depend on $m$, one has $\L_{OU} \varphi = 0$, thus the term of order $\epsilon^{-2}$ in~\eqref{eq:perturbed-test-fn_eqn_expanded-diff} vanishes. Define
\begin{equation} \label{eq:def_phi-1-diff}
    \varphi_1(x,m) \doteq m \frac{\sigma(x)}{f(x)} \cdot \nabla_x \varphi(x).
\end{equation}
Then it is straightforward to check that $\L_1 \varphi + \L_{OU} \varphi_1=0$, thus the term of order $\epsilon^{-1}$ in~\eqref{eq:perturbed-test-fn_eqn_expanded-diff} vanishes.

It remains to construct the function $\varphi_2$ such that the term of order $1$ in~\eqref{eq:perturbed-test-fn_eqn_expanded-diff} is equal to $\L\varphi$. Define, for all $x\in\mathbb{T}^d$ and $m\in\R$,
\begin{gather*}
    \L \varphi(x)=\int_\R (\L_0 \varphi + \L_1 \varphi_1)(x,m) d\nu^x(m),\\
    \vartheta(x,m)=(\L_0 \varphi + \L_1 \varphi_1)(x,m)-\L\varphi(x),
\end{gather*}
where we recall that $\nu^x = \mathcal{N}(0,\frac{f(x)h(x)^2}{2})$ is the invariant distribution of the ergodic Ornstein-Uhlenbeck process associated to $\L_{OU}$ on $\R$, for any fixed $x\in\mathbb{T}^d$.

Let $x\in\mathbb{T}^d$, then the Poisson equation $-\L_{OU}\varphi_2(x,\cdot)=\vartheta(x,\cdot)$ admits a solution $\varphi_2$, since the centering condition $\int \vartheta(x,m)d\nu^x(m)=0$ is satisfied. Precisely, one has the expressions
\begin{align}
\vartheta(x,m)&=- \left( \left| m \right|^2 - \frac{fh^2}{2} \right) \sigma \cdot \nabla_x \left( \frac{\sigma}{f} \cdot \nabla_x \varphi \right),\nonumber\\
\varphi_2(\cdot,m) &\doteq \frac{\left| m \right|^2}{2} \frac{\sigma}{f} \cdot \nabla_x \left( \frac{\sigma}{f} \cdot \nabla_x \varphi \right). \label{eq:def_phi-2-diff}
\end{align}
With the functions $\varphi_1$ and $\varphi_2$ constructed above, the multiscale expansion~\eqref{eq:perturbed-test-fn_eqn_expanded-diff} is rewritten as
\[
\L^\epsilon \varphi^\epsilon = \L \varphi + \epsilon \left( \L_0 \varphi_1 + \L_1 \varphi_2 \right) + \epsilon^2 \L_0 \varphi_2,
\]
which gives~\eqref{eq:generator_convergence-diff}, more precisely
\[
\underset{x\in\mathbb{T}^d}\sup~| \L^\epsilon \varphi^\epsilon (x,m) - \L \varphi (x) | \leq C_\varphi \left( \epsilon |m| + \epsilon^2 |m|^2 \right),
\]
for some constant $C_\varphi$ depending only on $\varphi$ and on the coefficients of the SDE.

It remains to check that $\L \varphi(x)=\int_\R (\L_0 \varphi + \L_1 \varphi_1)(x,m) d\nu^x(m)$ gives the expression~\eqref{eq:limit_generator-diff}. This concludes the identification of the limiting generator $\L$ using the perturbed test function method. The remaining ingredients of this strategy to prove the convergence in distribution of the process $X^\epsilon$ to the solution $X$ of the limiting equation associated with the limiting generator $\L$ follows from standard arguments which are omitted.

\section*{Acknowledgments}
The work of C.-E.~B. is partially supported by the following projects operated by the French National Research Agency: ADA (ANR-19-CE40-0019-02 ), BORDS (ANR-16-CE40-0027-01) and SIMALIN (ANR-19-CE40-0016).


\end{document}